\documentclass[oneside,reqno]{amsart}
\usepackage[latin9]{inputenc}
\usepackage{xcolor}
\usepackage{tipa}
\usepackage{amstext}
\usepackage{amsthm}
\usepackage{amssymb}
\usepackage{setspace}
\usepackage{esint}
\PassOptionsToPackage{normalem}{ulem}
\usepackage{ulem}
\usepackage[unicode=true,pdfusetitle,
 bookmarks=true,bookmarksnumbered=false,bookmarksopen=false,
 breaklinks=false,pdfborder={0 0 0},backref=false,colorlinks=true]
 {hyperref}
\usepackage{breakurl}

\makeatletter

\providecolor{lyxadded}{rgb}{0,0,1}
\providecolor{lyxdeleted}{rgb}{1,0,0}

\numberwithin{equation}{section}
\numberwithin{figure}{section}
\usepackage{enumitem}		
\newlength{\lyxlistindent}      
\theoremstyle{plain}
\newtheorem{thm}{\protect\theoremname}
  \theoremstyle{plain}
  \newtheorem{lem}[thm]{\protect\lemmaname}
  \theoremstyle{plain}
  \newtheorem{prop}[thm]{\protect\propositionname}
  \theoremstyle{definition}
  \newtheorem{defn}[thm]{\protect\definitionname}
  \theoremstyle{remark}
  \newtheorem{rem}[thm]{\protect\remarkname}
  \theoremstyle{plain}
  \newtheorem{cor}[thm]{\protect\corollaryname}
  \theoremstyle{definition}
  \newtheorem{example}[thm]{\protect\examplename}

\usepackage{graphics}
\usepackage{epstopdf}
\usepackage[all]{xy}
\usepackage{amscd}
\PassOptionsToPackage{normalem}{ulem}
\setlist[enumerate]{leftmargin=*,label=(\roman*),align=left}

\newcommand{\xyR}[1]{ \makeatletter
\xydef@\xymatrixrowsep@{#1} \makeatother} 
\newcommand{\xyC}[1]{ \makeatletter
\xydef@\xymatrixcolsep@{#1} \makeatother} 
\entrymodifiers={++[ ][F]} 

\newcommand{\ra}{\longrightarrow}

\newcommand{\field}[1]{\mathbb{#1}}
\newcommand{\R}{\field{R}} 
\newcommand{\N}{\field{N}} 

\newcommand{\D}{\mathcal{D}}

\newcommand{\eps}{\varepsilon} 
\renewcommand{\phi}{\varphi} 
\newcommand{\diff}[1]{\,\hbox{\rm d}#1} 

\newcommand{\Coo}{\mbox{\ensuremath{\mathcal{C}}}^{\infty}} 
\newcommand{\Gcinf}{\mathcal{G}\cinfty} 
\newcommand{\GD}{\mathcal{GD}}

\newcommand{\Rtil}{\widetilde \R} 
\newcommand{\Ctil}{\widetilde \CC} 
\newcommand{\otilc}{\widetilde \Omega_c} 
\newcommand{\gs}{{\mathcal{G}^s}} 
\newcommand{\ns}{\mathcal{N}^s} 

\newcommand{\sint}[1]{\langle#1\rangle} 
\newcommand{\Eball}{B^{{\scriptscriptstyle \text{\rm E}}}} 
\newcommand{\fcmp}{\Subset_{\text{f}}}
\newcommand{\exterior}[1]{{\rm ext}(#1)}

\newcommand{\then}{\quad \Longrightarrow \quad} 

\newcommand{\CC}{\mathbb C}

\newcommand{\supp}{\mathrm{supp}}
\newcommand{\cinfty}{{\mathcal C}^\infty}
\newcommand{\vphi}{\varphi}
\newcommand{\comp}{\Subset}
\newcommand{\esm}{{\mathcal E}^s_M}

\newcommand{\Om}{\Omega}

 \newcommand{\indlim}{{\displaystyle\lim_{\longrightarrow}\, }}

\makeatother

  \providecommand{\corollaryname}{Corollary}
  \providecommand{\definitionname}{Definition}
  \providecommand{\examplename}{Example}
  \providecommand{\lemmaname}{Lemma}
  \providecommand{\propositionname}{Proposition}
  \providecommand{\remarkname}{Remark}
\providecommand{\theoremname}{Theorem}

\begin{document}

\title{A convenient notion of compact set for generalized functions}

\author{Paolo Giordano \and Michael Kunzinger}

\thanks{P.~Giordano has been supported by grant P25116 and P25311 of the
Austrian Science Fund FWF}

\address{\textsc{University of Vienna, Austria}}

\email{paolo.giordano@univie.ac.at }

\thanks{M.~Kunzinger has been supported by grants P23714 and P25326 of the
Austrian Science Fund FWF}

\address{\textsc{University of Vienna, Austria}}

\email{michael.kunzinger@univie.ac.at}


\begin{abstract}
We introduce the notion of functionally compact sets into the theory
of nonlinear generalized functions in the sense of Colombeau. The
motivation behind our construction is to transfer, as far as possible,
properties enjoyed by standard smooth functions on compact sets into
the framework of generalized functions. Based on this concept, we
introduce spaces of compactly supported generalized smooth functions
that are close analogues to the test function spaces of distribution
theory. We then develop the topological and functional analytic foundations
of these spaces. 
\end{abstract}

\maketitle

\section{Introduction}

A main advantage of nonlinear generalized functions in the sense of
Colombeau as compared to Schwartz distributions is the fact that they
can be viewed as set-theoretic maps on domains consisting of generalized
points. This change of perspective allows to develop several branches
of the theory in close analogy to classical analysis, and thereby
has become increasingly important in recent years (cf., e.g., \cite{OK,AJ,AFJ,Gar05,Gar05b,ObVe08,AFJ09,GV,Ar-Fe-Ju-Ob12,GKV}).
In particular, appropriate topologies on spaces of nonlinear generalized
functions, the so called sharp topologies (see below for the definition),
have been introduced in \cite{S,S0} and have since been studied by
many authors. Apart from their central position in the structure theory
of Colombeau algebras, they also supply the foundation for applications
in the theory of nonlinear partial differential equations (e.g., for
a suitable concept of well-posedness).

From the point of view of analysis, a key notion underlying many existence
results is that of compactness. It turns out, however, that sharply
compact subsets of generalized points display certain unwanted properties:
e.g., no infinite subset of $\R^{n}$ is sharply compact since the
trace of the sharp topology on subsets of $\R^{n}$ is discrete. In
fact, this is a necessary consequence of the set $\Rtil$ of generalized
numbers containing actual infinitesimals, hence seems unavoidable
also in alternative approaches, cf.~\cite[Prop. 2.1]{GK2} and \cite[Thm. 25]{GKV}.

The importance of a convenient notion of compactness for nonlinear
generalized functions has been recognized by several authors, most
recently in \cite{ACJ}. The approach we take in the present paper
is to introduce an appropriate concept of compactly supported generalized
function, and then to study spaces consisting of such functions, which,
in analogy to the test function space $\mathcal{D}(\Omega)$ in distribution
theory, we denote by $\GD(U)$. The domain $U$ here is a set of generalized
points. Based on Garetto's theory of locally convex $\Ctil$-modules
\cite{Gar05,Gar05b,Gar09} we then develop the topological and functional
analytic foundations of these spaces. We find that they are indeed
close analogues of the classical spaces of test functions in that
they are countable strict inductive limits of complete metric spaces
$\GD_{K}(U)$ (analogues of $\mathcal{D}_{K}(\Omega)$ in distribution
theory) satisfying properties paralleling those of the classical strict
(LF)-spaces $\mathcal{D}(\Omega)$.

The plan of the paper is as follows: In the remainder of this introduction
we fix some basic notions used throughout this work. Section \ref{sec2}
introduces what we call \emph{functionally compact sets}, based on
work by Oberguggenberger and Vernaeve in \cite{ObVe08}. Building
on this, in Section \ref{sec3} we define compactly supported generalized
smooth functions (GSF), as well as the corresponding spaces $\GD(U)$
and $\GD_{K}(U)$. We also show that every Colombeau generalized function
$f\in\gs(\Omega)$ (in particular, every Schwartz distribution) defines
a compactly supported GSF $\bar{f}:\Rtil\ra\Rtil$ that coincides
with $f$ on $\otilc$. In order to obtain appropriate topologies
on these spaces, we define so-called generalized norms in Section
\ref{sec4}. These are maps that share the basic properties of classical
norms, yet take values in $\Rtil$, thereby generalizing a standard
alternative description of the sharp topology on generalized numbers
(cf.~\cite{AJ,GK2}). In Sections \ref{sec5} and \ref{sec6} these
generalized norms are employed to endow the spaces $\GD_{K}(U)$ with
metric topologies. In particular, in Section \ref{sub:5.1} we study
connections between non-Archimedean properties and Hausdorff topological
vector spaces of generalized functions, proving an impossibility theorem:
there does not exist a Hausdorff topological vector subspace of the
Colombeau special algebra which contains the Dirac delta and even
a single trace of an open set of the sharp topology. The completeness
of the spaces $\GD_{K}(U)$ is established in Section \ref{sec7}.
In the final Section \ref{sec8} we derive the fundamental functional
analytic properties of the space $\GD(U)$.

\subsection{Basic notions}

Our main references for Colombeau's theory are \cite{C1,C2,MObook,GKOS}.
The special Colombeau algebra $\gs(\Omega)$ over an open subset $\Omega$
of $\R^{n}$ is defined as the quotient $\esm(\Omega)/\ns(\Omega)$,
where (setting $I:=(0,1]$ and noting that in the naturals $\N=\{0,1,2,3\ldots\}$
we include zero.) 
\[
\begin{split}\esm(\Om) & :=\{(u_{\eps})\in\cinfty(\Omega)^{I}\mid\forall K\comp\Om\,\forall\alpha\in\N^{n}\,\exists N\in\N:\sup_{x\in K}|\partial^{\alpha}u_{\eps}(x)|=O(\eps^{-N})\}\\
\ns(\Om) & :=\{(u_{\eps})\in\cinfty(\Omega)^{I}\mid\forall K\comp\Om\,\forall\alpha\in\N^{n}\,\forall m\in\N:\sup_{x\in K}|\partial^{\alpha}u_{\eps}(x)|=O(\eps^{m})\}.
\end{split}
\]
Elements of $\esm(\Omega)$ are called moderate, those of $\ns(\Omega)$
are called negligible. Nets in $\esm(\Omega)$ are written as $(u_{\eps})$,
and $u=[u_{\eps}]$ denotes the corresponding equivalence class in
$\gs(\Om)$. For $(u_{\eps})\in\ns(\Om)$ we also write $(u_{\eps})\sim0$.
We will abbreviate `Colombeau generalized function' by CGF. $\gs(-)$
is a fine sheaf of differential algebras and there exist sheaf embeddings
(based on smoothing via convolution) of the space of Schwartz distributions
$\D'$ into $\gs$ (cf.~\cite{GKOS}).

Given $\Omega\subseteq\R^{n}$ open, the space of generalized points
in $\Omega$ is $\widetilde{\Omega}=\Omega_{M}/\sim$, where $\Omega_{M}=\{(x_{\eps})\in\Omega^{I}\mid\exists N\in\N:|x_{\eps}|=O(\eps^{-N})\}$
is called the set of moderate nets and $(x_{\eps})\sim(y_{\eps})$
if $|x_{\eps}-y_{\eps}|=O(\eps^{m})$ for every $m\in\N$. In the
particular case $\Omega=\R$ we obtain the ring of Colombeau generalized
numbers (CGN) $\Rtil=\R_{M}/\sim$ (and analogously for $\Ctil$),
which can also be written as $\Rtil=\R_{M}/\ns$, where $\ns$ is
the set of all negligible nets of real numbers $(x_{\eps})\in\R^{I}$,
i.e. such that $(x_{\eps})\sim0$. $\Rtil$ is an ordered ring with
respect to its natural order relation: $x\le y$ iff there are representatives
$(x_{\eps})$ and $(y_{\eps})$ such that $x_{\eps}\le y_{\eps}$
for $\eps$ sufficiently small. We point out that, in the present
work, the notion $x>y$ does \emph{not} mean $x\ge y$ and $x\not=y$.
Rather, it is to be understood as $x-y\ge0$ and $x-y$ invertible.
By \cite[1.2.38]{GKOS} and \cite[Prop. 3.2]{M2} we have: 
\begin{lem}
\label{lem:mayer} Let $x\in\Rtil$. Then the following are equivalent: 
\begin{enumerate}
\item $x>0$. 
\item For each representative $(x_{\eps})$ of $x$ there exists some $\eps_{0}$
and some $m$ such that $x_{\eps}>\eps^{m}$ for all $\eps<\eps_{0}$. 
\item For each representative $(x_{\eps})$ of $x$ there exists some $\eps_{0}$
such that $x_{\eps}>0$ for all $\eps<\eps_{0}$. 
\end{enumerate}
\end{lem}
We shall use the notation $\diff{\eps^{m}}:=[\eps^{m}]\in\Rtil$ for
any $m\in\R$. Hence $x>0$ is equivalent to $x\ge\diff{\eps^{m}}$
for some $m>0$. If $\mathcal{P}(\eps)$ is a property of $\eps\in I$,
we will also sometimes use the notation $\forall^{0}\eps:\ \mathcal{P}(\eps)$
to denote $\exists\eps_{0}\in I\,\forall\eps\in(0,\eps_{0}]:\ \mathcal{P}(\eps)$.

The space of compactly supported generalized points $\otilc$ is defined
by $\Omega_{c}/\sim$, where $\Omega_{c}:=\{(x_{\eps})\in\Omega^{I}\mid\exists K\comp\Omega\,\forall^{0}\eps:\ x_{\eps}\in K\}$
and $\sim$ is the same equivalence relation as in the case of $\widetilde{\Omega}$.

Concerning intervals, we use the following notations: $[a,b]:=\{x\in\Rtil\mid a\le x\le b\}$,
$[a,b]_{\R}:=[a,b]\cap\R$. Also, for $x,y\in\Rtil^{n}$ we write
$x\approx y$ if $x-y$ is infinitesimal, i.e.\ if $|x-y|\le r$
for all $r\in\R_{>0}$.

As already indicated above, the natural topology for Colombeau-type
spaces is the so-called sharp topology (\cite{Biag,S0,S,AJ,AJOS,M,GV}).
This topology is generated by balls $B_{\rho}(x)=\{y\in\Rtil^{n}\mid|y-x|<\rho\}$,
where $|-|$ is the natural extension of the Euclidean norm to $\Rtil^{n}$,
$|[x_{\eps}]|:=[|x_{\eps}|]\in\Rtil$, and $\rho\in\Rtil_{>0}$ is
strictly positive (\cite{AFJ,AFJ09,GK2}). For Euclidean balls, we
will write $\Eball_{\rho}(x)=\{y\in\R^{n}\mid|y-x|<\rho\}$. On the
other hand, the so-called Fermat-topology on $\Rtil^{n}$ (see \cite{GK2,GKV})
is generated by the balls $B_{r}(x)$ for $x\in\Rtil^{n}$ and $r\in\R_{>0}$.
Originally, the sharp topology was introduced using an ultrametric
as follows: The map 
\begin{align*}
v & :\R_{M}\longrightarrow(-\infty,\infty]\\
v & ((u_{\eps})):=\sup\{b\in\R\mid|u_{\eps}|=O(\eps^{b})\}.
\end{align*}
gives a pseudovaluation on $\Rtil$. Then setting$|-|_{e}:\Rtil\to[0,\infty)$,
$|u|_{e}:=\exp(-v(u))$ provides a translation-invariant complete
ultrametric 
\begin{align*}
d & _{s}:\Rtil\times\Rtil\longrightarrow\R_{+}\\
d & _{s}(u,v):=|u-v|_{e}
\end{align*}
on $\Rtil$, which induces the sharp topology on $\Rtil$.

Garetto in \cite{Gar05,Gar05b} extended the above construction to
arbitrary locally convex spaces by functorially assigning a space
of CGF $\mathcal{G}_{E}$ to any given locally convex space $E$.
In this approach, the seminorms of $E$ are used to define pseudovaluations
which induce a generalized locally convex topology on the $\Ctil$-module
$\mathcal{G}_{E}$, again called sharp topology. In the present paper,
we will exclusively work with $\Rtil$-modules. We note, however,
that all our constructions trivially carry over to the $\Ctil$-case.

For any $S\subseteq I$, $e_{S}$ denotes the equivalence class in
$\Rtil$ of the characteristic function of $S$ (cf.~\cite{AJ,Ver10}).
Any $e_{S}$ is an idempotent, and $e_{S}+e_{S^{c}}=1$. Also, $e_{S}\not=0$
if and only if $0\in\overline{S}$. For any subset $A$ of $\Rtil^{n}$,
its interleaving (cf.~\cite{ObVe08}) is defined as 
\[
\text{interl}(A):=\left\{ \sum_{j=1}^{m}e_{S_{j}}a_{j}\mid m\in\N,\{S_{1},\dots,S_{m}\}\text{ a partition of }I,\ a_{j}\in A\right\} .
\]

If $(A_{\eps})$ is a net of subsets of $\R^{n}$ then the internal
set (\cite{ObVe08,Ver11}) generated by $(A_{\eps})$ is 
\[
[A_{\eps}]=\left\{ [x_{\eps}]\in\Rtil^{n}\mid x_{\eps}\in A_{\eps}\text{ for }\eps\text{ small}\right\} ,
\]
and the strongly internal set (\cite{GKV}) generated by $(A_{\eps})$
is 
\[
\langle A_{\eps}\rangle:=\left\{ [x_{\eps}]\in\Rtil^{n}\mid x_{\eps}\in_{\eps}A_{\eps}\right\} .
\]
Here, $x_{\eps}\in_{\eps}A_{\eps}$ means that $x_{\eps}\in A_{\eps}$
for $\eps$ small and that the same property holds for any representative
of $[x_{\eps}]$. The net $(A_{\eps})$ is called sharply bounded
if there exists some $N\in\R_{>0}$ such that for $\eps$ sufficiently
small we have $\sup_{x\in A_{\eps}}|x|\le\eps^{-N}$. Equivalently,
we have that $(A_{\eps})$ is sharply bounded if there exists $\rho\in\Rtil_{>0}$
such that $[A_{\eps}]\subseteq B_{\rho}(0)$.

Finally, given $X\subseteq\Rtil^{n}$ and $Y\subseteq\Rtil^{d}$,
then (see \cite{GKV}) 
\[
f:X\longrightarrow Y\text{ is a \emph{generalized smooth function} (GSF)}
\]
if there exists a net $u_{\eps}\in\cinfty(\Omega_{\eps},\R^{d})$
defining $f$ in the sense that $X\subseteq\langle\Omega_{\eps}\rangle$,
$f([x_{\eps}])=[u_{\eps}(x_{\eps})]\in Y$ and $(\partial^{\alpha}u_{\eps}(x_{\eps}))\in\R_{M}^{d}$
for all $x=[x_{\eps}]\in X$ and all $\alpha\in\N^{n}$. The space
of GSF from $X$ to $Y$ is denoted by $\Gcinf(X,Y)$ (in contrast
to \cite{GKV}, where the notation $\widetilde{\mathcal{G}}(X,Y)$
was used). GSF are a natural generalization of CGF to general domains.
In particular, for any $\Omega\subseteq\R^{n}$ open, $\Gcinf(\otilc)\simeq\gs(\Omega)$.
GSF on subsets of $\Rtil^{n}$, endowed with the sharp topology, form
a sub-category of the category of topological spaces. In particular,
they can be composed unrestrictedly.

\section{\label{sec2}A new notion of compact subset for nonlinear generalized
functions}

Even though the intervals $[a,b]\subseteq\Rtil$, $a$, $b\in\R$,
are neither compact in the sharp nor in the Fermat topology (see \cite[Thm. 25]{GKV}),
analogously to the case of smooth functions, a GSF satisfies an extreme
value theorem on such sets. In fact, we have: 
\begin{prop}
\label{prop:extremeValues}Let $f\in\Gcinf(X,\Rtil)$ be a generalized
smooth function defined on the subset $X$ of $\Rtil^{n}$. Let $\emptyset\ne K=[K_{\eps}]\subseteq X$
be an internal set generated by a sharply bounded net $(K_{\eps})$
of compact sets $K_{\eps}\comp\R^{n}$ , then 
\begin{equation}
\exists m,M\in K\,\forall x\in K:\ f(m)\le f(x)\le f(M).\label{eq:epsExtreme}
\end{equation}
\end{prop}
\begin{proof}
By \cite[Lem. 28]{GKV}, $f$ can be represented by a net $u_{\eps}\in\Coo(\R^{n},\R^{d})$.
Since $K\ne\emptyset$, for $\eps$ sufficiently small, say for $\eps\in(0,\eps_{0}]$,
$K_{\eps}$ is non-empty and, by assumption, it is also compact. For
all $\eps\in(0,\eps_{0}]$ we have 
\[
\exists m_{\eps},M_{\eps}\in K_{\eps}\,\forall x\in K_{\eps}:\ u_{\eps}(m_{\eps})\le u_{\eps}(x)\le u_{\eps}(M_{\eps}).
\]
Since the net $(K_{\eps})$ is sharply bounded, both the nets $(m_{\eps})$
and $(M_{\eps})$ are moderate. Therefore $m=[m_{\eps}]$, $M=[M_{\eps}]\in K\subseteq X$.
Take any $x\in[K_{\eps}]$, then there exists a representative $(x_{\eps})$
such that $x_{\eps}\in K_{\eps}$ for $\eps$ small. Therefore $f(m)=[u_{\eps}(m_{\eps})]\le[u_{\eps}(x{}_{\eps})]=f(x)\le f(M)$. 
\end{proof}
We shall use the assumptions on $K$ and $(K_{\eps})$ given in this
theorem to introduce a new notion of ``compact subset'' which behaves
better than the usual classical notion of compactness in the sharp
topology. 
\begin{defn}
\label{def:functCmpt} A subset $K$ of $\Rtil^{n}$ is called \emph{functionally
compact}, denoted by $K\fcmp\Rtil^{n}$, if there exists a net $(K_{\eps})$
such that 
\begin{enumerate}
\item \label{enu:defFunctCmpt-internal}$K=[K_{\eps}]\subseteq\Rtil^{n}$ 
\item \label{enu:defFunctCmpt-sharpBound}$(K_{\eps})$ is sharply bounded 
\item \label{enu:defFunctCmpt-cmpt}$\forall\eps\in I:\ K_{\eps}\Subset\R^{n}$ 
\end{enumerate}
If, in addition, $K\subseteq U\subseteq\Rtil^{n}$ then we write $K\fcmp U$.
Finally, we write $[K_{\eps}]\fcmp U$ if \ref{enu:defFunctCmpt-sharpBound},
\ref{enu:defFunctCmpt-cmpt} and $[K_{\eps}]\subseteq U$ hold. 
\end{defn}
We note that in \ref{enu:defFunctCmpt-cmpt} it suffices to ask that
$K_{\eps}$ be closed since it is bounded by \ref{enu:defFunctCmpt-sharpBound},
at least for $\eps$ small. In fact, we have: 
\begin{lem}
\noindent \label{lem:equdef}A subset $K$ of $\Rtil^{n}$ is functionally
compact if and only if it is internal and sharply bounded.\end{lem}
\begin{proof}
\noindent By \cite[Lemma 2.4 and Cor. 2.2]{ObVe08}, every sharply
bounded internal set $K$ has a sharply bounded representative $(K_{\eps})$
consisting of closed (hence compact) subsets of $\R^{n}$. 
\end{proof}
\noindent We motivate the name \emph{functionally compact subset}
by anticipating that on this type of subsets, GSF have properties
very similar to those that ordinary smooth functions have on standard
compact sets. 
\begin{rem}
\noindent \label{rem:defFunctCmpt}\ 
\begin{enumerate}
\item \label{enu:rem-defFunctCmpt-closed}By \cite[Prop. 2.3]{ObVe08},
any internal set $K=[K_{\eps}]$ is closed in the sharp topology.
In particular, the open interval $(0,1)\subseteq\Rtil$ is not functionally
compact since it is not closed. 
\item \label{enu:rem-defFunctCmpt-ordinaryCmpt}If $H\Subset\R^{n}$ is
a non-empty ordinary compact set, then $\widetilde{H}=[H]$ is functionally
compact. In particular, $[0,1]=\widetilde{[0,1]_{\R}}=\left[[0,1]_{\R}\right]$
is functionally compact. 
\item \label{enu:rem-defFunctCmpt-empty}The empty set $\emptyset=\widetilde{\emptyset}\fcmp\Rtil$. 
\item \label{enu:rem-defFunctCmpt-equivDef}By Lemma \ref{lem:equdef},
$\Rtil^{n}$ is not functionally compact since it is not sharply bounded. 
\item \label{enu:rem-defFunctCmpt-cmptlySuppPoints}The set of compactly
supported points $\Rtil_{c}$ is not functionally compact because
the GSF $f(x)=x$ does not satisfy the conclusion \eqref{eq:epsExtreme}
of Prop.~\ref{prop:extremeValues}. 
\end{enumerate}
\end{rem}
We start the study of functionally compact sets by proving suitable
generalizations of theorems from classical analysis. 
\begin{thm}
\label{thm:image}Let $K\subseteq X\subseteq\Rtil^{n}$, $f\in\Gcinf(X,\Rtil^{d})$.
Then $K\fcmp\Rtil^{n}$ implies $f(K)\fcmp\Rtil^{d}$.\end{thm}
\begin{proof}
Let $(K_{\eps})$ be as in Def.~\ref{def:functCmpt} and let the
GSF $f$ be defined by the net $u_{\eps}\in\Coo(\R^{n},\R^{d})$.
Let us first prove that $f(K)=[u_{\eps}(K_{\eps})]$. In fact, $y\in f(K)=f([K_{\eps}])$
is equivalent to 
\begin{equation}
\exists(x_{\eps})\in\R_{M}^{n}\,\forall^{0}\eps:\ x_{\eps}\in K_{\eps}\text{ and }y=[u_{\eps}(x_{\eps})].\label{eq:1equiv-f(K)}
\end{equation}
This necessary entails $y\in[u_{\eps}(K_{\eps})]$. Vice versa, if
$y\in[u_{\eps}(K_{\eps})]$, then there exists $(y_{\eps})\in\R_{M}^{d}$
such that $y_{\eps}\in u_{\eps}(K_{\eps})$ for $\eps$ small. Hence,
for each of these $\eps$ there also exists $x_{\eps}\in K_{\eps}$
such that $y_{\eps}=u_{\eps}(x_{\eps})$, which implies $y=[u_{\eps}(x_{\eps})]$,
i.e.\ \eqref{eq:1equiv-f(K)} since $(K_{\eps})$ is sharply bounded.
Clearly, $u_{\eps}(K_{\eps})\Subset\R^{d}$, so it remains to prove
that the net $(u_{\eps}(K_{\eps}))$ is sharply bounded. If $\forall\eps_{0}\,\exists\eps\le\eps_{0}:\ K_{\eps}=\emptyset$,
then $[K_{\eps}]=K=\emptyset$, so $f(K)=\emptyset$ and the conclusion
is trivial. Otherwise, assume that $K_{\eps}\ne\emptyset$ for $\eps\le\eps_{0}$
and proceed by contradiction assuming that 
\begin{equation}
\forall k\in\N\,\exists(\eps_{kn})_{n}\downarrow0\,\forall n\,\exists y_{kn}\in u_{\eps_{kn}}(K_{\eps_{kn}}):\ |y_{kn}|>\eps_{kn}^{-k}.\label{eq:HpContr-image}
\end{equation}
We can write $y_{kn}=u_{\eps_{kn}}(x_{kn})$ for some $x_{kn}\in K_{\eps_{kn}}$.
Next, set $\eps_{0}:=\eps_{00}$ and for $k>0$ pick $n_{k}$ such
that $\eps_{kn_{k}}<\min\left(\frac{1}{k},\eps_{k-1}\right)$ and
set $\eps_{k}:=\eps_{kn_{k}}$. Take any $\bar{x}_{\eps}\in K_{\eps}$
for each $\eps\le\eps_{0}$, and set $x_{\eps}:=x_{kn_{k}}$ if $\eps=\eps_{k}$
and $x_{\eps}:=\bar{x}_{\eps}$ otherwise. Then $x_{\eps}\in K_{\eps}$
for $\eps\le\eps_{0}$, so $x=[x_{\eps}]\in K\subseteq X$ and $(u_{\eps}(x_{\eps}))\in\R_{M}^{d}$
by the definition of GSF, which contradicts \eqref{eq:HpContr-image}. 
\end{proof}
As a corollary of this theorem and Rem.\ \eqref{rem:defFunctCmpt}.\ref{enu:rem-defFunctCmpt-ordinaryCmpt}
we get 
\begin{cor}
\label{cor:intervalsFunctCmpt}If $a$, $b\in\Rtil$ and $a\le b$,
then $[a,b]\fcmp\Rtil$. 
\end{cor}
Let us note that $a$, $b\in\Rtil$ can also be infinite, e.g.~$a=[-\eps^{-N}]$,
$b=[\eps^{-M}]$ or $a=[\eps^{-N}]$, $b=[\eps^{-M}]$ with $M>N$. 
\begin{lem}
\label{thm:unionIntersection}Let $K$, $H\fcmp\Rtil^{n}$, then we
have: 
\begin{enumerate}
\item \label{enu:unionIntersection-interleaving}$K\cup H\subseteq\text{\emph{interl}}(K\cup H)\fcmp\Rtil^{n}$ 
\item \label{enu:unionIntersection-union}If $K\cup H$ is internal, then
it is functionally compact 
\item \label{enu:unionIntersection-intersection}If $K\cap H$ is internal,
then it is functionally compact. 
\end{enumerate}
\end{lem}
\begin{proof}
\noindent \ref{enu:unionIntersection-interleaving} follows from \cite[Prop. 2.8]{ObVe08}
which implies $K\cup H\subseteq\text{interl}(K\cup H)=[K_{\eps}\cup H_{\eps}]$,
where the nets $(K_{\eps})$ and $(H_{\eps})$ satisfy Def.~\ref{def:functCmpt}.
Property \ref{enu:unionIntersection-union} follows from \cite[Lemma 2.7]{ObVe08}
which implies that if the union of internal sets is internal, then
it is equal to its interleaving. Property \ref{enu:unionIntersection-intersection}
is a consequence of Lemma \ref{lem:equdef}. 
\end{proof}
\noindent If $H\subseteq K\fcmp\Rtil^{n}$, then also $H$ is sharply
bounded. So, another consequence of Lemma \ref{lem:equdef} is the
following: 
\begin{cor}
\noindent \label{cor:subsets}Let $H\subseteq K\fcmp\Rtil^{n}$, then
$H$ internal implies $H\fcmp\Rtil^{n}$. 
\end{cor}
\noindent Finally, in the following result we consider the product
of functionally compact sets: 
\begin{prop}
\noindent \label{prop:product}Let $K\fcmp\Rtil^{n}$ and $H\fcmp\Rtil^{d}$,
then $K\times H\fcmp\Rtil^{n+d}$. In particular, if $a_{i}\le b_{i}$
for $i=1,\ldots,n$, then $\prod_{i=1}^{n}[a_{i},b_{i}]\fcmp\Rtil^{n}$.\end{prop}
\begin{proof}
\noindent From \cite[Prop. 2.13]{ObVe08} if follows that $K\times H$
is internal, in fact for $K=[K_{\eps}]$, $H=[H_{\eps}]$, $K\times H=[K_{\eps}\times H_{\eps}]$.
From this representation it immediately follows that $H\times K$
is sharply bounded as well, so Lemma \ref{lem:equdef} gives the claim. 
\end{proof}

\section{\label{sec3}Compactly supported generalized smooth functions}

Our main goal in this section is to define and study an analogue within
GSF of the space $\D_{K}(\Omega)$ of smooth functions supported in
a fixed compact set $K\comp\Omega$. In order to define this space,
we first try to define the concept of support of a GSF. Clearly, if
$\phi\in\D_{[-a,a]_{\R}}(\R)$, one would expect that $\phi$ should
have compact support also if we think of $\phi$ as a GSF. In fact,
$\text{supp}(\phi)$ should be contained in $[-a,a]\fcmp\Rtil$. Already
this basic requirement implies that if $f\in\Gcinf(X,Y)$, the natural
definition 
\[
S(f):=X\setminus\bigcup\left\{ B_{\rho}(x)\cap X\mid x\in X,\ \rho\in\Rtil_{>0},\ f|_{{B_{\rho}(x)\cap X}}=0\right\} 
\]
doesn't fit with our intuition. Indeed, if we take the aforementioned
$\phi$ so that $\phi(0)=1$, and $S\subseteq(0,1]$ such that $0\in\bar{S}$
and $0\in\overline{S^{c}}$, then $\phi(e_{S}\diff{\eps}^{-1})=e_{S^{c}}\ne0$
and $e_{S}\diff{\eps}^{-1}\in S(\phi)\setminus[-a,a]$. This motivates
the following 
\begin{defn}
\label{def:support}Let $X\subseteq\Rtil^{n}$, $Y\subseteq\Rtil^{d}$
and $f\in\Gcinf(X,Y)$, then 
\[
\text{supp}(f):=\overline{\left\{ x\in X\mid|f(x)|>0\right\} },
\]
where here $\overline{(-)}$ denotes the relative closure in $X$
with respect to the sharp topology. 
\end{defn}
Using this concept, we have $\text{supp}(\phi)\subseteq[-a,a]$. In
fact, $|\phi(x)|=[|\phi(x_{\eps})|]>0$ implies $|\phi(x_{\eps})|>\eps^{q}$
for some $q\in\R_{>0}$ and for $\eps$ small, and hence $x_{\eps}\in[-a,a]_{\R}$. 
\begin{rem}
\noindent \label{rem:support}\ 
\begin{enumerate}
\item \label{enu:supportGSFandCGF}In the setting of Colombeau algebras,
one usually defines the support of some $f\in\gs(\Omega)$ as a subset
of $\Omega\subseteq\R^{n}$, i.e., as a set of classical points, namely
as $\supp_{\gs}(f):=\Omega\setminus\bigcup\left\{ B_{r}(x)\cap\Omega\mid x\in\Omega,\ r\in\R_{>0},\ f|_{B_{r}(x)\cap\Omega}=0\right\} $,
where the last equality has to be understood in $\gs(B_{r}(x)\cap\Omega)$.
Using $X=\widetilde{\Omega}_{c}$ as the natural domain of any $f\in\gs(\Omega)$
(cf.~\cite[Thm. 37]{GKV}), it is then immediate that $\supp(f)\cap\Omega\subseteq\supp_{\gs}(f)$. 
\item Let $u\in\D'(\Omega)$ be a Schwartz distribution and denote by $\iota:\D'(\Omega)\to\Gcinf(\otilc,\R)$
a standard embedding via convolution. Then $\supp(\iota(u))\cap\Omega\subseteq\supp(u)$,
as follows from \ref{enu:supportGSFandCGF} and the fact that $\iota$
is a sheaf-morphism. 
\item Assume that the embedding $\iota:\D'(\Omega)\to\Gcinf(\otilc,\R)$
has been defined by using a mollifier $\rho\in\mathcal{S}(\R^{n})$
which is identically equal to 1 in the ball $\Eball_{p}(0)$, $p\in\R_{>0}$.
Then $\delta(x)=\diff{\eps}^{-n}$ for each $x\in B_{p\cdot\text{d}\eps}(0)$
and hence $B_{p\cdot\text{d}\eps}(0)\subseteq\text{supp}(\iota(\delta))$,
whereas $\text{supp}(\iota(\delta))\cap\R^{n}=\{0\}$. 
\item \noindent In general, $\text{supp}(f)$ is not an internal set because
it is not generally closed by finite interleaving (see \cite[Lem.~2.7]{ObVe08}).
Consider e.g.~$X$ with only near standard points and $f\in\Gcinf(X,\Rtil)$
which is strictly positive on two disjoint intervals. However, if
$X$ itself is closed under finite interleaving then so is $\text{supp}(f)$.
\end{enumerate}
\end{rem}
If $(u_{\eps})$ defines $f\in\Gcinf(X,Y)$, the internal set $\left[\text{supp}(u_{\eps})\right]$
is not intrinsically defined since it depends on the defining net
$(u_{\eps})$. Consider, e.g., $u_{\eps}(x):=\phi(x)+\eps^{1/\eps}>0$
where $\phi\in\Coo(\R,\R_{\ge0})$.

In our further analysis we will repeatedly make use of the following
notion: 
\begin{defn}
For $A\subseteq\Rtil^{n}$ we call the set 
\[
\exterior{A}:=\{x\in\Rtil^{n}\mid\forall a\in A:\ |x-a|>0\}
\]

\noindent the \emph{strong exterior} of $A$.
\end{defn}
This set can also be described in the following way: 
\begin{lem}
\label{lem:charactExtWithe_S}If $A\subseteq\Rtil^{n}$, then $\exterior{A}=\{x\in\Rtil^{n}\mid\forall S\subseteq I:\ e_{S}\not=0\ \Rightarrow\ xe_{S}\not\in Ae_{S}\}$.\end{lem}
\begin{proof}
$\subseteq$: Let $x=[x_{\eps}]\in\exterior{A}$ and suppose that
there exists some $S\subseteq I$ with $e_{S}\not=0$ and some $a=[a_{\eps}]\in A$
such that $xe_{S}=ae_{S}$. Then there exists some $q>0$ such that
$|x_{\eps}-a_{\eps}|>\eps^{q}$ for $\eps$ small. However, $xe_{S}=ae_{S}$
implies that $|x_{\eps}-a_{\eps}|=O(\eps^{q+1})$ for $\eps\to0$,
$\eps\in S$, a contradiction.

$\supseteq$: If there exists some $a=[a_{\eps}]\in A$ such that
$|x-a|\not>0$ then there is a sequence $\eps_{k}\downarrow0$ with
$|x_{\eps_{k}}-a_{\eps_{k}}|<\eps_{k}^{k}$ for all $k$. Letting
$S:=\{\eps_{k}\mid k\in\N\}$ implies $xe_{S}=ae_{S}$. 
\end{proof}
For non-trivial internal sets we have the following characterization
of the strong exterior: 
\begin{lem}
\label{lem:exteriorAndKc}Let $\emptyset\not=[K_{\eps}]=K\subseteq\Rtil^{n}$.
Then 
\[
\exterior{K}=\sint{K_{\eps}^{c}}.
\]
\end{lem}
\begin{proof}
Suppose first that $x\in\sint{K_{\eps}^{c}}$ and let $S\subseteq I$
with $e_{S}\not=0$. Suppose that there existed some $a\in K$ with
$xe_{S}=ae_{S}$. Since $a\in K$ there exists a representative $(a_{\eps})$
of $a$ with $a_{\eps}\in K_{\eps}$ for all $\eps$. Then $xe_{S}=ae_{S}$
implies that there exists a representative $(x_{\eps})$ of $x$ and
a sequence $\eps_{k}\downarrow0$ in $S$ with $x_{\eps_{k}}=a_{\eps_{k}}\in K_{\eps_{k}}$
for all $k$. But this contradicts the fact that $x\in\sint{K_{\eps}^{c}}$.

Conversely, if $x\not\in\sint{K_{\eps}^{c}}$ then there exists a
representative $(x_{\eps})$ of $x$ and a sequence $\eps_{m}\downarrow0$
with $x_{\eps_{m}}\in K_{\eps_{m}}$ for all $m$. Since $K\not=\emptyset$,
there exists some $w=[w_{\eps}]\in K$ with $w_{\eps}\in K_{\eps}$
for all $\eps$. Now let 
\[
a_{\eps}:=\begin{cases}
x_{\eps_{m}} & \text{ if }\eps=\eps_{m}\\
w_{\eps} & \text{ otherwise }
\end{cases}
\]
and set $S:=\{\eps_{m}\mid m\in\N\}$. Then $a=[a_{\eps}]\in[K_{\eps}]$
and $xe_{S}=ae_{S}$ by construction. Thus $xe_{S}\in Ke_{S}$, and
so $x\not\in\exterior{K}$. 
\end{proof}
As an immediate conclusion we obtain: 
\begin{cor}
\label{lem:exteriorDoesNotDepOnK_eps} Let $K=[K_{\eps}]=[L_{\eps}]\ne\emptyset$.
Then $\sint{K_{\eps}^{c}}=\sint{L_{\eps}^{c}}$. 
\end{cor}
The next result relates the support of a GSF to the exterior of certain
internal sets. To formulate it concisely, we introduce the following
notations: Denote by $\mathcal{K}_{f}$ the set of all internal $\emptyset\ne K\subseteq\Rtil^{n}$
with $\exterior{K}\not=\emptyset$ and such that there exists a net
$u_{\eps}\in\Coo(\R^{n},\R^{d})$ that defines $f$ and such that
$\left[u_{\eps}(x_{\eps})\right]=0$ for all $[x_{\eps}]\in\text{\exterior{K}}$.
Also, denote by $\mathcal{H}_{f}$ the set of all the internal sets
of the form $K=\left[\text{supp}(u_{\eps})\right]\subseteq\Rtil^{n}$
for some net $u_{\eps}\in\Coo(\R^{n},\R^{d})$ that defines $f$ and
such that both $K$ and $\exterior{K}$ are non empty.

\noindent Then we have: 
\begin{lem}
\label{lem:suppAndSupp-f_eps}Let $X\subseteq\Rtil^{n}$, $Y\subseteq\Rtil^{d}$
and $f\in\Gcinf(X,Y)$. Then 
\begin{equation}
\text{\emph{supp}}(f)\subseteq X\cap\bigcap_{K\in\mathcal{K}_{f}}K\subseteq X\cap\bigcap_{K\in\mathcal{H}_{f}}K.\label{eq:inclusionsSupp}
\end{equation}
\end{lem}
\begin{proof}
Since $X\cap\bigcap_{K\in\mathcal{K}_{f}}K$ is a sharply closed subset
of $X$, in order to show the first inclusion in \eqref{eq:inclusionsSupp},
it suffices to prove that 
\[
\left\{ x\in X\mid|f(x)|>0\right\} \subseteq X\cap\bigcap_{K\in\mathcal{K}_{f}}K.
\]
Let $x\in X$ be such that $|f(x)|>0$, so that 
\begin{equation}
\exists r\in\R_{>0}:\ |f(x)|>\diff{\eps}^{r}.\label{eq:abs_f_diffr}
\end{equation}
Let $K=[K_{\eps}]\in\mathcal{K}_{f}$, and assume, by contradiction,
that $x=[x_{\eps}]\notin K$. We first prove that 
\begin{equation}
\exists q\in\N_{>r}\,\exists(\eps_{k})_{k\in\N}\downarrow0\,\forall k\in\N:\ \Eball_{\eps_{k}^{q}}(x_{\eps_{k}})\subseteq K_{\eps_{k}}^{c},\label{eq:lemma2}
\end{equation}
where $r$ comes from \eqref{eq:abs_f_diffr}. In fact, suppose to
the contrary that 
\[
\forall q\in\N_{>r}\,\exists\eps_{q}\,\forall\eps\le\eps_{q}\,\exists y_{\eps}^{(q)}\in\Eball_{\eps^{q}}(x_{\eps}):\ y_{\eps}^{(q)}\in K_{\eps}.
\]
We may assume that $(\eps_{q})_{q\in\N}\downarrow0$. Setting $\tilde{y}_{\eps}:=y_{\eps}^{(q)}$
for $\eps\in(\eps_{q+1},\eps_{q}]$, we have $x=[\tilde{y}_{\eps}]\in K$,
which contradicts $x\notin K$.

By assumption $\exists z=[z_{\eps}]\in\exterior{K}$, and hence 
\[
\exists s\in\N_{>q}\,\forall^{0}\eps:\ d(z_{\eps},K_{\eps})>\eps^{s},
\]
where $q$ comes from \eqref{eq:lemma2}. Using $(\eps_{k})_{k\in\N}$
from \eqref{eq:lemma2}, we set $\tilde{x}_{\eps}:=x_{\eps_{k}}$
if $\eps=\eps_{k}$ and $\tilde{x}_{\eps}:=z_{\eps}$ otherwise. Then
$\tilde{x}:=[\tilde{x}_{\eps}]\in\exterior{K}$. But $K\in\mathcal{K}_{f}$,
so there exists a net $u_{\eps}\in\Coo(\R^{n},\R^{d})$ that defines
$f$ and such that $[u_{\eps}(\tilde{x}_{\eps})]=0$. In particular,
$|u_{\eps_{k}}(\tilde{x}_{\eps_{k}})|=|u_{\eps_{k}}(x_{\eps_{k}})|=O(\eps_{k}^{2s})$
as $k\to+\infty$, which contradicts $|f(x)|=[|u_{\eps}(x_{\eps})|]>\diff{\eps}^{r}>\diff{\eps}^{s}$.

Turning now to the second inclusion, assume that $x\in X\setminus\bigcap_{K\in\mathcal{H}_{f}}K$,
so that there exists a net $(u_{\eps})$ that defines $f$, with $\emptyset\ne K:=[\text{supp}(u_{\eps})]\subseteq\Rtil^{n}$
and $\exterior{K}\ne\emptyset$, but such that $x\notin K$. We want
to show that $x\notin\bigcap_{K'\in\mathcal{K}_{f}}K'$. But for all
$[y_{\eps}]\in\exterior{K}=\sint{\text{supp}(u_{\eps})^{c}}$, we
have $y_{\eps}\notin\text{supp}(u_{\eps})$ for $\eps$ small. Hence,
$u_{\eps}(y_{\eps})=0$ for these $\eps$, and this yields $[u_{\eps}(y_{\eps})]=0$.
This proves that $K\in\mathcal{K}_{f}$, but $x\notin K$.\end{proof}
\begin{rem}
\ 
\begin{enumerate}
\item Using methods from Nonstandard Analysis, one can prove the converse
of the first inclusion in \eqref{eq:inclusionsSupp} in the following
case: If $X=\Rtil^{n}$ and the sharp interior of $\left\{ x\in\Rtil^{n}\mid f(x)=0\right\} $
is non-empty, then 
\[
\text{{supp}}(f)=\bigcap_{K\in\mathcal{K}_{f}}K.
\]
To see this, let $x\in\bigcap_{K\in\mathcal{K}_{f}}K$ and choose some $z$
in the interior of $\left\{ y\in\Rtil^{n}\right.$ $\left. \mid f(y)=0\right\} $ and
some $q>0$ such that $B_{d\eps^{q}}(z)\subseteq\left\{ y\in\Rtil^{n}\mid f(y)=0\right\} $.
Given a representative $(v_{\eps})$ of $f$, let $(\chi_{\eps})$
be a moderate net of smooth functions such that $\chi_{\eps}$ vanishes
on $\Eball_{\eps^{q+2}}(z_{\eps})$ and is identically equal to $1$
on $\R^{n}\setminus\Eball_{\eps^{q+1}}(z_{\eps})$. Then setting $u_{\eps}:=\chi_{\eps}v_{\eps}$
gives a new representative of $f$ such that $[\{y\in\R^{n}\mid u_{\eps}(y)=0\}]$
has non-empty sharp interior. Let $m_{\eps}\in\N$, $m_{\eps}\to\infty$
as $\eps\to0$. Then set $K_{\eps}:=\{y\in\R^{n}\mid|u_{\eps}(y)|\ge\eps^{m_{\eps}}\}$
and $K:=[K_{\eps}]$. It follows that $\emptyset\not=\exterior{K}\subseteq\{y\in\R^{n}\mid f(y)=0\}$,
so $x\in K$. Fix any $k\in\N$. Now using nonstandard notation, letting
$\rho:=[\eps]$, $u:=[u_{\eps}]$, and fixing a representative of
$x$, which we temporarily simply denote by $x$, the above in particular
implies 
\[
^{*}\N_{\infty}\subseteq\{m\in{}^{*}\N\mid\exists y_{k}\in{}^{*}\R^{n}:|y_{k}-x|\le\rho^{k}\wedge|u(y_{k})|\ge\rho^{m}\}.
\]
By the underspill principle, therefore, there also exists some $m\in\N$
and some $y_{k}\in{}^{*}\R^{n}$ such that $|y_{k}-x|\le\rho^{k}$
and $|u(y_{k})|\ge\rho^{m}$. Taking equivalence classes of these
nets, it follows that there exist  $y_{k}\in\Rtil^{n}$ such
that $|y_{k}-x|\le\diff{\eps}^{k}$ and $|f(y_{k})|>0$. Since $y_{k}\to x$
in the sharp topology, this implies that $x\in\text{supp}(f)$. 
\item The assumption $\exterior{K}\not=\emptyset$ in the definition of
$\mathcal{K}_{f}$ is essential. To illustrate this, take $\vphi$
as defined before Def.\ \ref{def:support}, pick any sequence $\eps_{k}\downarrow0$
and set $K_{\eps}:=\{0\}$ for $\eps=\eps_{k}$, and $K_{\eps}=\R^{n}$
otherwise. Then $\exterior{K}=\emptyset$, and obviously $\mathrm{supp}(\vphi)\not\subseteq K$. 
\item The question of whether the reverse of the last inclusion in \eqref{eq:inclusionsSupp}
also holds remains open. 
\end{enumerate}
\end{rem}
\noindent In the following section we shall see that even though the
notion of support introduced above may not be entirely satisfactory,
there nevertheless is a very convenient notion of being compactly
supported for GSF.

\subsection{The spaces \texorpdfstring{$\mathcal{G}\D_{K}(U,Y)$}{GDK(U,Y)}
and \texorpdfstring{$\mathcal{G}\D(U,Y)$}{GD(U,Y)}}

A frequently used idea to solve problems like the previous ones comes
by considering a family of GSF having ``good representatives'',
i.e.~possessing a defining net $(u_{\eps})$ that conforms to our
intuition and includes the examples we have in mind. In the following,
we denote by $(u^{1},\ldots,u^{d})$ the components of a function
$u$ which takes values, e.g., in $\R^{d}$. 
\begin{defn}
\label{def:cmptlySuppGSF}Let $\emptyset\ne K\fcmp U\subseteq\Rtil^{n}$
and $Y\subseteq\Rtil^{d}$, then we say that $f$ is a \emph{GSF compactly
supported in }$K$, and we write 
\[
f\in\GD_{K}(U,Y)
\]
if $f\in\Gcinf(U,Y)$ and there exists a net $(u_{\eps})$ such that: 
\begin{enumerate}
\item \label{enu:def-cmptlySuppGSF-K_epsSharplyBound} $(u_{\eps})$ defines
$f$, where $u_{\eps}\in\Coo(\R^{n},\R^{d})$ for all $\eps$. 
\item \label{enu:def-cmptlySuppGSF-u_eps-K_eps-relations}$\forall\alpha\in\N^{n}\,\forall[x_{\eps}]\in\exterior{K}:\ \left[\partial^{\alpha}u_{\eps}(x_{\eps})\right]=0$. 
\end{enumerate}
\noindent Moreover, we set 
\[
\GD(U,Y):=\bigcup_{\emptyset\ne K\fcmp U}\GD_{K}(U,Y).
\]
We will simply use the symbols $\GD_{K}(U)$ and $\GD(U)$ if $Y=\Rtil$.\end{defn}
\begin{rem}
\label{rem:defCmptlySuppGSF}\ 
\begin{enumerate}
\item \label{enu:suppSubsetK}Lemma \ref{lem:suppAndSupp-f_eps} implies
that if $f\in\GD_{K}(U,Y)$, then $\text{supp}(f)\subseteq K$ because
$K\in\mathcal{K}_{f}$. The converse implication for an arbitrary
subset $U$ remains an open problem. For the case $U=\Rtil^{n}$,
and if $\text{supp}(f)$ is not empty, see Thm.~\ref{thm:globallyDefGSFandSupport}
below.
\item It is clear that, in general, another net defining $f:U\ra\Rtil^{d}$
will not necessarily satisfy \ref{enu:def-cmptlySuppGSF-u_eps-K_eps-relations}
of Def.~\ref{def:cmptlySuppGSF} because such a net is not bound
to have any particular behavior outside of $U$. 
\item \label{enu:counterexamples}Set $K_{\eps}:=\overline{\Eball_{1}(0)}$
and $L_{\eps}:=K_{\eps}\setminus\Eball_{e^{-1/\eps}}(0)$, then for
the Hausdorff distance of $K_{\eps}$ and $L_{\eps}$ we obtain $d_{\mathcal{H}}(K_{\eps},L_{\eps})=e^{-\frac{1}{\eps}}$.
By \cite[Cor. 2.10]{ObVe08}, it follows that $[K_{\eps}]=[L_{\eps}]$.
If we consider a net of smooth functions such that $u_{\eps}|_{\Eball_{1/2}(0)}=1$,
$u_{\eps}|_{\R^{2}\setminus K_{\eps}}=0$, $u_{\eps}\ge0$, then $\sup_{y\in\R^{2}\setminus K_{\eps}}|\partial^{\alpha}u_{\eps}(y)|=0$
but $\sup_{y\in\R^{2}\setminus L_{\eps}}|\partial^{\alpha}u_{\eps}(y)|=1$.
This motivates the use of the strongly internal set $\sint{K_{\eps}^{c}}$
in \ref{enu:def-cmptlySuppGSF-u_eps-K_eps-relations} instead of the
simpler $[K_{\eps}^{c}]$. Analogously, we can consider as $L_{\eps}$
an $e^{-\frac{1}{\eps}}$-mesh of points for $K_{\eps}=\overline{\Eball_{1}(0)}$.
$L_{\eps}$ may also contain points in $K_{\eps}^{c}$, but so that
$d(x,K_{\eps})=e^{-\frac{1}{\eps}}$. This example shows clearly that
we need to be sufficiently ``far'' from $\partial K_{\eps}$ to
be sure that $\left[\partial^{\alpha}u_{\eps}(x_{\eps})\right]=0$,
i.e.\ at points $x=[x_{\eps}]$ such that $\left[d(x_{\eps},K_{\eps})\right]>0$,
as stated in \ref{enu:def-cmptlySuppGSF-u_eps-K_eps-relations}. Since
$\exterior{K}$ is independent of the choice of representative of
$K$ by Cor.\ \ref{lem:exteriorDoesNotDepOnK_eps}, so is Def.~\ref{def:cmptlySuppGSF}.
This is essential to prove the completeness of the spaces $\GD_{K}(U)$
and $\GD(U)$. 
\end{enumerate}
\end{rem}
The following result will turn out to be useful when proving results
by contradiction in several instances below. It permits to restrict
the analysis to only two cases: points in $K$ or in $\exterior{K}$.
To state it more clearly, we say that a generalized point $[y_{\eps}]$
\emph{joins points of the sequence} $(y_{k})_{k\in\N}$ \emph{at }$(\eps_{k})_{k\in\N}$
if $\forall N\in\N\,\exists k\ge N:\ y_{\eps_{k}}=y_{k}$. 
\begin{lem}
\label{lem:twoCases}Let $K=[K_{\eps}]\fcmp\Rtil^{n}$, $(\eps_{k})_{k\in\N}$
a sequence in $(0,1]$ which strictly decreases to $0$, and $(y_{k})_{k\in\N}$
a sequence in $\R^{n}$. For each $k\in\N$, let $x_{k}\in K_{\eps_{k}}$
be such that $d(y_{k},x_{k})=d(y_{k},K_{\eps_{k}})$. Then either 
\begin{enumerate}
\item \label{enu:Icase}$\exists[y_{\eps}]\in\exterior{K}:\ [y_{\eps}]$
joins points of the sequence $(y_{k})_{k\in\N}$ at $(\eps_{k})_{k\in\N}$ 
\end{enumerate}
\noindent or 
\begin{enumerate}
\item \label{enu:IIcase}$\exists[\bar{y}_{\eps}],[\bar{x}_{\eps}]$ joining
points of the sequences $(y_{k})_{k\in\N}$ and $(x_{k})_{k\in\N}$
(respectively) at $(\eps_{k})_{k\in\N}$ such that $[\bar{y}_{\eps}]=[\bar{x}_{\eps}]$
and $\bar{x}_{\eps}\in K_{\eps}\ \forall\eps$. 
\end{enumerate}
\end{lem}
\begin{proof}
We can always pick a point $e_{\eps}\in K_{\eps}^{c}$ so that $e:=[e_{\eps}]\in\exterior{K}$;
in fact, since $(K_{\eps})$ is sharply bounded, we can find $e_{\eps}\in\R^{n}\setminus K_{\eps}$
so that $d(e_{\eps},K_{\eps})>1$ and $(e_{\eps})$ is moderate. We
can also take a point $i_{\eps}\in K_{\eps}$ for each $\eps$ because,
without loss of generality, we can assume that $K_{\eps}\ne\emptyset$
for all $\eps$.

The first alternative in the statement is realized if 
\begin{equation}
\exists b\in\R_{>0}\,\exists N\in\N\,\forall k\ge N:\ \left|y_{k}-x_{k}\right|>\eps_{k}^{b}.\label{eq:Icase}
\end{equation}
Set $y_{\eps}:=y_{k}$ if $\eps=\eps_{k}$ and $y_{\eps}:=e_{\eps}$
otherwise. Then, if $\eps=\eps_{k}$, by \eqref{eq:Icase} we have
$d(y_{\eps},K_{\eps})=\left|y_{k}-x_{k}\right|>\eps^{b}$; otherwise
$d(y_{\eps},K_{\eps})=d(e_{\eps},K_{\eps})>1\ge\eps^{b}$. Therefore,
$[y_{\eps}]\in\sint{K_{\eps}^{c}}=\exterior{K}$ and, of course, $[y_{\eps}]$
joins points of the sequence $(y_{k})_{k\in\N}$ at $(\eps_{k})_{k\in\N}$.

Vice versa, if 
\begin{equation}
\forall h\in\N\,\exists k_{h}>h:\ \left|y_{k_{h}}-x_{k_{h}}\right|\le\eps_{k_{h}}^{h},\label{eq:II case}
\end{equation}
then we can set $\bar{y}_{\eps}:=y_{k_{h}}$, $\bar{x}_{\eps}:=x_{k_{h}}$
if $\eps=\eps_{k_{h}}$ and $\bar{y}_{\eps}:=\bar{x}_{\eps}:=i_{\eps}$
otherwise. Then, if $\eps=\eps_{k_{h}}$, by \eqref{eq:II case} we
have $\left|\bar{y}_{\eps}-\bar{x}_{\eps}\right|=\left|y_{k_{h}}-x_{k_{h}}\right|\le\eps_{k_{h}}^{h}=\eps^{h}$;
otherwise $\left|\bar{y}_{\eps}-\bar{x}_{\eps}\right|=\left|i_{\eps}-i_{\eps}\right|=0$,
and \ref{enu:IIcase} follows. 
\end{proof}
We use the above result to guarantee that the maximum values of any
partial derivative $\partial^{\alpha}f$ are attained on $K$ and
not outside, as precisely stated in the following 
\begin{lem}
\label{lem:supK_eps-RsetminusK_eps} Let $(u_{\eps})$ and $K=[K_{\eps}]$
satisfy Def.~\ref{def:cmptlySuppGSF}, then 
\begin{equation}
\forall\alpha\in\N^{n}\,\forall i=0,\ldots,n:\ \left[\sup_{y\in\R^{n}}|\partial^{\alpha}u_{\eps}^{i}(y)|\right]=\left[\sup_{x\in K_{\eps}}|\partial^{\alpha}u_{\eps}^{i}(x)|\right].\label{eq:supK_eps-RsetminusK_eps}
\end{equation}
\end{lem}
\begin{proof}
By contradiction, assume that 
\begin{equation}
\left[\sup_{y\in\R^{n}}|\partial^{\alpha}u_{\eps}^{i}(y)|\right]\ne\left[\sup_{x\in K_{\eps}}|\partial^{\alpha}u_{\eps}^{i}(x)|\right].\label{eq:HpAbs}
\end{equation}
For simplicity of notation, set $v_{\eps}:=\partial^{\alpha}u_{\eps}^{i}$.
Inequality \eqref{eq:HpAbs} means 
\[
\exists a\in\R_{>0}\,\exists\eps_{k}\searrow0\,\forall k\in\N:\ \left|\sup_{y\in\R^{n}}|v_{\eps_{k}}(y)|-\sup_{x\in K_{\eps}}|v_{\eps_{k}}(x)|\right|>\eps_{k}^{a}.
\]
Thus 
\begin{equation}
\forall k\in\N\,\exists y_{k}\in\R^{n}:\ \eps_{k}^{a}+\sup_{x\in K_{\eps_{k}}}\left|v_{\eps_{k}}(x)\right|<\left|v_{\eps_{k}}(y_{k})\right|.\label{eq:y_kFromPropFondSup}
\end{equation}
We can hence apply Lemma \ref{lem:twoCases}. In the first case \ref{enu:Icase}
we have $y:=[y_{\eps}]\in\exterior{K}$ so that $[v_{\eps}(y_{\eps})]=0$
by Def.~\ref{def:cmptlySuppGSF} \ref{enu:def-cmptlySuppGSF-u_eps-K_eps-relations}.
Since $[y_{\eps}]$ joins points of $(y_{k})_{k\in\N}$ at $(\eps_{k})_{k\in\N}$,
from \eqref{eq:y_kFromPropFondSup} and $[v_{\eps}(y_{\eps})]=0$
we get 
\[
\eps_{k}^{a}\le\eps_{k}^{a}+\sup_{x\in K_{\eps_{k}}}\left|v_{\eps_{k}}(x)\right|<\left|v_{\eps_{k}}(y_{k})\right|=\left|v_{\eps_{k}}(y_{\eps_{k}})\right|<\eps_{k}^{a+1},
\]
for $k$ sufficiently big, which gives a contradiction in the first
case. In the second case, i.e., \ref{enu:IIcase} of Lemma \ref{lem:twoCases},
we have $\bar{y}:=[\bar{y}_{\eps}]=\bar{x}:=[\bar{x}_{\eps}]\in K\subseteq U$.
Therefore $[v_{\eps}(\bar{y}_{\eps})]=[v_{\eps}(\bar{x}_{\eps})]$
by the definition of GSF. Since $[\bar{y}_{\eps}]$ and $[\bar{x}_{\eps}]$
join points of $(y_{k})_{k\in\N}$ and $(x_{k})_{k\in\N}$, respectively,
at $(\eps_{k})_{k\in\N}$, for $k$ sufficiently big we have 
\begin{align*}
\eps_{k}^{a}+\sup_{x\in K_{\eps_{k}}}\left|v_{\eps_{k}}(x)\right| & <\left|v_{\eps_{k}}(y_{k})\right|\le\left|v_{\eps_{k}}(y_{k})-v_{\eps_{k}}(x_{k})\right|+\left|v_{\eps_{k}}(x_{k})\right|\\
 & \le\eps_{k}^{a+1}+\sup_{x\in K_{\eps_{k}}}\left|v_{\eps_{k}}(x)\right|,
\end{align*}
leading to a contradiction also in the second case. 
\end{proof}
The previous result will be essential to prove that any compactly
supported GSF can be extended to the whole of $\Rtil^{n}$, and to
define an $\Rtil$-valued norm of $f$ that does not depend on $K$.

\noindent Using Lemma \ref{lem:supK_eps-RsetminusK_eps}, we can prove
that any derivative of a compactly supported GSF is globally bounded
in an appropriate sense: 
\begin{lem}
\label{lem:cmptlySuppGSFareBounded}Let the net $(u_{\eps})$ satisfies
Def.~\ref{def:cmptlySuppGSF}, then 
\[
\forall\alpha\in\N^{n}\,\exists C\in\Rtil\,\forall\beta\in\N^{n}:\ |\beta|\le|\alpha|\Rightarrow\left[\sup_{y\in\R^{n}}|\partial^{\beta}u_{\eps}(y)|\right]\le C.
\]
\end{lem}
\begin{proof}
\noindent By the extreme value property Prop.~\ref{prop:extremeValues}
we can set $C_{i\beta}:=\left[\sup_{y\in K_{\eps}}|\partial^{\beta}u_{\eps}^{i}(y)|\right]$,
where $\beta\in\N^{n}$, $|\beta|\le|\alpha|$, and $i=1,\ldots,n$.
Set $C:=1+\sqrt{n}\max_{\substack{|\beta|\le|\alpha|\\
1\le i\le n
}
}C_{i\beta}$. Then $C>C_{i\beta}$ and property \eqref{eq:supK_eps-RsetminusK_eps}
yields $\left[\sup_{y\in\R^{n}}|\partial^{\beta}u_{\eps}^{i}(y)|\right]$
$=\left[\sup_{x\in K_{\eps}}|\partial^{\beta}u_{\eps}^{i}(x)|\right]=C_{i\beta}$
and hence $\left[\sup_{y\in\R^{n}}|\partial^{\beta}u_{\eps}(y)|\right]<C$. 
\end{proof}

Moreover, compactly supported GSF can be extended in a unique way
to the entire $\Rtil^{n}$: 
\begin{thm}
\label{thm:extensionCmptlySupportedGSF}Let $\emptyset\ne K\fcmp U\subseteq\Rtil^{n}$
and $f\in\GD_{K}(U,\Rtil^{d})$ be defined by $(u_{\eps})$ which
satisfies Def.~\ref{def:cmptlySuppGSF}. Then $(u_{\eps})$ defines
a GSF of the type $\Rtil^{n}\ra\Rtil^{d}$ and there exists one and
only one $\bar{f}\in\GD_{K}(\Rtil^{n},\Rtil^{d})$ such that: 
\begin{enumerate}
\item $\bar{f}|_{K}=f|_{K}$ 
\item $\bar{f}|_{\text{\emph{ext}}(K)}=0.$ 
\end{enumerate}
\noindent Moreover, this $\bar{f}$ satisfies 
\begin{enumerate}[resume]
\item \label{enu:extension-partialDer}If U is sharply open and $\alpha\in\N^{n}$,
then $\partial^{\alpha}\bar{f}|_{U}=\partial^{\alpha}f$ \end{enumerate}
\begin{enumerate}
\item \label{enu:extension-CGF}$\bar{f}|_{\Rtil_{c}^{n}}$ can be identified
with a Colombeau generalized function. 
\end{enumerate}
\end{thm}
\begin{proof}
The existence part follows by showing that for all $\alpha\in\N^{n}$
and all $[x_{\eps}]\in\Rtil^{n}$ we have $\left(\partial^{\alpha}u_{\eps}(x_{\eps})\right)\in\R_{M}^{d}$.
This follows since Lemma \ref{lem:cmptlySuppGSFareBounded} yields
that for any $\alpha\in\N^{n}$ there exists some $(C_{\alpha\eps})\in\R_{M}^{n}$
such that $|\sup_{x\in\R^{n}}\partial^{\alpha}u_{\eps}(x)|\le C_{\alpha\eps}$
for all $\eps$ small. To prove uniqueness, let $g\in\GD_{K}(\Rtil^{n},\Rtil^{d})$
be such that $g|_{K}=\bar{f}|_{K}=f$ and $g|_{\exterior{K}}=\bar{f}|_{\exterior{K}}=0$.
By contradiction, assume that $\bar{f}(y)=[u_{\eps}(y_{\eps})]\ne g(y)=[v_{\eps}(y_{\eps})]$,
for some $y=[y_{\eps}]\in\Rtil^{n}$, where $(v_{\eps})$ defines
$g$. Thus 
\begin{equation}
\exists a\in\R_{>0}\,\exists\eps_{k}\searrow0\,\forall k\in\N:\ \left|u_{\eps_{k}}(y_{\eps_{k}})-v_{\eps_{k}}(y_{\eps_{k}})\right|>\eps_{k}^{a}.\label{eq:fbarNe-gAty}
\end{equation}
By Lemma \ref{lem:twoCases}, this leaves two possibilities. In the
first one, there exists a point $z=[z_{\eps}]\in\exterior{K}$ which
joins points of the sequence $(y_{\eps_{k}})_{k\in\N}$ at $(\eps_{k})_{k\in\N}$.
Therefore $g(z)=[v_{\eps}(z_{\eps})]=\bar{f}(z)=[u_{\eps}(z_{\eps})]=0$,
which gives a contradiction at $\eps=\eps_{k}$ when compared with
\eqref{eq:fbarNe-gAty}. In the second one, there exists a point $\bar{z}=[\bar{z}_{\eps}]\in K$
joining points of the sequence $(y_{\eps_{k}})_{k\in\N}$ at $(\eps_{k})_{k\in\N}$.
Once again, we have $g(\bar{z})=[v_{\eps}(\bar{z}_{\eps})]=\bar{f}(\bar{z})=[u_{\eps}(\bar{z}_{\eps})]=f(z)$,
in contradiction to \eqref{eq:fbarNe-gAty}.

Furthermore, \cite[Thm. 31]{GKV} implies claim \ref{enu:extension-partialDer}.
Finally, \ref{enu:extension-CGF} follows from the isomorphism $\Gcinf(\Rtil_{c},\Rtil^{d})\simeq\gs(\R)^{d}$. 
\end{proof}
We also have this simple but useful result: 
\begin{lem}
\label{lem:derivativeOfCmptlySuppGSF}Let $\emptyset\ne K\fcmp U\subseteq\Rtil^{n}$,
$U$ be a sharply open set, $f\in\GD_{K}(U,\Rtil^{d})$ and $\alpha\in\N^{n}$.
Then $\partial^{\alpha}f\in\GD_{K}(U,\Rtil^{d})$. 
\end{lem}
Thm.~\ref{thm:extensionCmptlySupportedGSF} opens the possibility
to restrict our attention to compactly supported GSF whose domain
is the entire $\Rtil^{n}$, as stated in the following 
\begin{thm}
\label{thm:globallyDefGDK}For $\emptyset\ne K\fcmp\Rtil^{n}$ and
$Y\subseteq\Rtil^{d}$ set 
\begin{equation}
\GD^{\text{\emph{g}}}(K,Y):=\left\{ f\in\Gcinf(\Rtil^{n},Y)\mid f|_{\exterior{K}}=0\right\} ,\label{eq:GDKGlobal}
\end{equation}
where the $\text{\emph{g}}$ superscript means \emph{globally defined.}
Let $K\subseteq U\subseteq\Rtil^{n}$, then 
\begin{enumerate}
\item \label{enu:globalSubGDK}If $f\in\GD^{\text{g}}(K,Y)$, then $f|_{U}\in\GD_{K}(U,Y)$. 
\item \label{enu:GDKSubGlobal}If $f\in\GD_{K}(U,\Rtil^{d})$, then $\exists!\bar{f}\in\GD^{\text{\emph{g}}}(K,\Rtil^{d}):\ \bar{f}|_{K}=f|_{K}$. 
\item \label{enu:globalEqualsGDK}$\GD^{\text{\emph{g}}}(K,Y)=\GD_{K}(\Rtil^{n},Y)$.
\end{enumerate}
\end{thm}
\begin{proof}
\ref{enu:globalSubGDK}: Let $(u_{\eps})$ be any net that defines
$f$, so that $u_{\eps}\in\Coo(\R^{n},\R^{d})$. Clearly $f|_{U}\in\Gcinf(U,Y)$,
so for all $\alpha\in\N^{n}$ and $x=[x_{\eps}]\in\exterior{K}$ it
remains to show that $[\partial^{\alpha}u_{\eps}(x_{\eps})]=0$. By
Lem.~\ref{lem:exteriorAndKc}, we get $\exterior{K}=\sint{K_{\eps}^{c}}$,
which is a sharply open set. So $x\in\exterior{K}$ yields $B_{r}(x)\subseteq\exterior{K}$
for some $r\in\Rtil_{>0}$, and hence $f|_{B_{r}(x)}=0$. Thereby
$\partial^{\alpha}f(x)=[\partial^{\alpha}u_{\eps}(x_{\eps})]=0$.

\ref{enu:GDKSubGlobal}: This is exactly Thm.~\ref{thm:extensionCmptlySupportedGSF}.

\ref{enu:globalEqualsGDK}: This follows directly from \ref{enu:globalSubGDK}
by setting $U=\Rtil^{n}$, and by Def.~\ref{def:cmptlySuppGSF}
which yields $f|_{\exterior{K}}=0$. 
\end{proof}
For the extension of property \ref{enu:GDKSubGlobal} to arbitrary
codomains $Y\subseteq\Rtil^{d}$ (provided that $U$ is strongly internal)
see Thm.~\ref{thm:extensionAndCodomains} below.

Note explicitly that in \eqref{eq:GDKGlobal}, instead of the more
technical properties of Def.~\ref{def:cmptlySuppGSF}, we have a
concise and simpler pointwise condition. Notwithstanding this and several
other positive aspects of definition \eqref{eq:GDKGlobal} (see the
second part of Lem.~\ref{lem:suppAndSupp-f_eps} and the following
Thm.~\ref{thm:globallyDefGSFandSupport}), in the present work, we
prefer not to change Def.~\ref{def:cmptlySuppGSF} in favor of \eqref{eq:GDKGlobal}:
On the one hand, Def.~\ref{def:cmptlySuppGSF} is nearer to the classical
definition of $\D_{K}(\Omega)$, where the domain is $\Omega\subseteq\R^{n}$;
on the other hand, for applications of these notions to geometry,
locally defined functions are a more natural setting. We can also
summarize these results by saying that compactly supported
GSF have a good notion of being compactly supported, and globally defined
compactly supported GSF are in addition well-behaved with respect to the concept
of support introduced above. This is also clearly confirmed by the following 
\begin{thm}
\label{thm:globallyDefGSFandSupport}Let $Y\subseteq\Rtil^{d}$, and
$f\in\Gcinf(\Rtil^{n},Y)$ such that $\text{\emph{supp}}(f)\ne\emptyset$.
Then $f\in\GD_{K}(\Rtil^{n},Y)$ if and only if $\text{\emph{supp}}(f)\subseteq K$.\end{thm}
\begin{proof} One implication is immediate from  
Remark \ref{rem:defCmptlySuppGSF}.\ref{enu:suppSubsetK}. 
Conversely, if $\text{{supp}}(f)\subseteq K$, 
by Thm.~\ref{thm:globallyDefGDK}.\ref{enu:globalEqualsGDK} we
have to show that $f|_{\exterior{K}}=0$. Let $x\in\exterior{K}$
but assume that $f(x)\ne0$. Pick any $y\in\left\{ y'\in\Rtil^{n}\mid|f(y')|>0\right\} $,
which is non-empty because $\text{supp}(f)\ne\emptyset$. Let $(u_{\eps})$
be a net that defines $f$ and let $y=[y_{\eps}]$.
Since $f(x)\ne0$,
there exists $S\subseteq I$ such that $e_S\not=0$ and $|f(x)e_S|>0$. 
Then setting $z:=xe_{S}+ye_{S^{c}}$, we have $|f(z)|>0$
and hence $z\in\text{supp}(f)\subseteq K$. But then $xe_{S}=ze_{S}\in Ke_{S}$
which, by Lem.~\ref{lem:charactExtWithe_S}, yields $x\notin\exterior{K}$,
a contradiction. 
\end{proof}

\subsection{Examples of compactly supported GSF}

A first class of examples comes by considering each $\phi\in\D_{K}(\Omega)$,
$K\Subset\Omega\subseteq\R^{n}$. Indeed, it suffices to set $K_{\eps}:=K$
and $u_{\eps}(x):=\phi(x)$ if $x\in\Omega$ and $u_{\eps}(x):=0$
otherwise to have that $\phi\in\GD_{\widetilde{K}}(\otilc,\Rtil)$,
where we recall that $\widetilde{K}=[K]$. Therefore $\D_{K}(\Omega)\subseteq\GD_{\widetilde{K}}(\otilc,\Rtil)$.

Moreover, since any given CGF can be defined by a net $(u_{\eps})$
of maps with sharply bounded compact supports, we have the following
result: 
\begin{thm}
\label{thm:CGFasCompctlySuppGSF}Let $\Omega$ be an open subset of
$\R^{n}$ and $J=[J_{\eps}]\in\Rtil$ be a CGN such that $\lim_{\eps\to0^{+}}J_{\eps}=+\infty$.
Set $K_{\eps}:=\{x\in\Omega\mid|x|\le J_{\eps}\}$ and $K:=[K_{\eps}]$.
Then for all $f\in\Gcinf(\otilc,\Rtil^{d})$ there exists $\bar{f}\in\GD_{K}(\Rtil^{n},\Rtil^{d})$
such that $\bar{f}|_{\otilc}=f$.\end{thm}
\begin{proof}
Set $U_{\eps}:=\{x\in\Omega\mid|x|<\frac{1}{2}J_{\eps}\}$ so that
$U_{\eps}\subseteq K_{\eps}$ for $\eps$ small. Let $\chi_{\eps}\in\Coo(\Omega,\R)$
be such that $\chi|_{U_{\eps}}=1$ and $\text{supp}(\chi_{\eps})\subseteq K_{\eps}$.
Let $f\in\Gcinf(\otilc,\Rtil^{d})$ be represented by $(v_{\eps})$,
with $v_{\eps}\in\Coo(\R^{n},\R^{d})$, and set $u_{\eps}:=\chi_{\eps}\cdot v_{\eps}$.
Then each $u_{\eps}$ is compactly supported and any $x=[x_{\eps}]\in\otilc$
satisfies $x_{\eps}\in U_{\eps}$ for $\eps$ small because $\lim_{\eps\to0^{+}}J_{\eps}=+\infty$.
Therefore $\bar{f}:=[u_{\eps}(-)]\in\GD_{K}(\Rtil^{n},\Rtil^{d})$,
and if $x_{\eps}\in U_{\eps}$ then $u_{\eps}(x_{\eps})=v_{\eps}(x_{\eps})$,
so $\bar{f}|_{\otilc}=f$. 
\end{proof}
This theorem gives an infinity of non-trivial examples of compactly
supported GSF. Moreover, even though $\bar{f}$ depends on the fixed
infinite number $J\in\Rtil$, every such $\bar{f}$ contains all the
information of the original CGF $f$ because $\bar{f}|_{\otilc}=f$.

Finally, the constant function $f(x)=1$ for all $x\in\Rtil$ is not
compactly supported. In fact, by contradiction, assume that $f$ admits
$(u_{\eps})$ and $(K_{\eps})$ such that Def.~\ref{def:cmptlySuppGSF}
holds. Then choosing $r$ large enough that $\diff{\eps}^{-r}\in\Rtil\setminus K$
we arrive at $f(\diff{\eps}^{-r})=[u_{\eps}(\eps^{-r})]=0$.

\section{\label{sec4}Generalized norms on \texorpdfstring{$\mathcal{G}\D_{K}$}{GDK}
and \texorpdfstring{$\GD$}{GD}}

As a first step to topologizing the spaces $\mathcal{G}\D_{K}$ and
$\GD$ we prove: 
\begin{thm}
\label{thm:GD_K-Rtil-Module}Let $\emptyset\ne K\fcmp U\subseteq\Rtil^{n}$,
then 
\begin{enumerate}
\item \label{enu:GD-RtilModule-module}$\GD_{K}(U,\Rtil^{d})$ is an $\Rtil$-module 
\item \label{enu:GD-RtilModule-increasing}For all non-empty $H\fcmp U$,
the inclusion $K\subseteq H$ implies $\GD_{K}(U,\Rtil^{d})\subseteq\GD_{H}(U,\Rtil^{d})$. 
\end{enumerate}
\end{thm}
\begin{proof}
\ref{enu:GD-RtilModule-module} is immediate from Def.~\ref{def:cmptlySuppGSF}.

\ref{enu:GD-RtilModule-increasing}: Take $f\in\GD_{K}(U,\Rtil^{d})$.
Since $K\subseteq H$, by \cite[Prop. 2.8]{ObVe08} for each representative
$(K_{\eps})$ of $K$ we get the existence of a representative $(H_{\eps})$
of $H$ such that $K_{\eps}\subseteq H_{\eps}$ for all $\eps$. Therefore,
$\text{ext}(K)=\sint{K_{\eps}^{c}}\supseteq\sint{H_{\eps}^{c}}=\text{ext}(H)$
and hence the conclusion follows. 
\end{proof}
From the extreme value property, Prop.~\ref{prop:extremeValues},
it is natural to expect that the following CGN could serve as generalized
$\Rtil$-valued norms. 
\begin{defn}
\label{def:generalizedNormsGD_K}Let $\emptyset\ne K\fcmp U\subseteq\Rtil^{n}$,
where $U$ is a sharply open set. Let $m\in\N$ and $f\in\GD_{K}(U,\Rtil^{d})$.
Then 
\[
\Vert f\Vert_{m,K}:=\max_{\substack{|\alpha|\le m\\
1\le i\le d
}
}\max\left(\left|\partial^{\alpha}f^{i}(M_{\alpha i})\right|,\left|\partial^{\alpha}f^{i}(m_{\alpha i})\right|\right)\in\Rtil,
\]
where $m_{\alpha i}$, $M_{\alpha i}\in K$ satisfy 
\[
\forall x\in K:\ \partial^{\alpha}f^{i}(m_{\alpha i})\le\partial^{\alpha}f^{i}(x)\le\partial^{\alpha}f^{i}(M_{\alpha i}).
\]

\end{defn}
The following result permits to calculate the (generalized) norm $\Vert f\Vert_{m,K}$
using any net $(v_{\eps})$ that defines $f$. In case the net $(v_{\eps})$
satisfies Def.~\ref{def:cmptlySuppGSF}, it also permits to prove
that this norm does not depend on $K$, as is the case for any ordinary
compactly supported smooth function.

\noindent Even though $\Vert f\Vert_{m,K}\in\Rtil$, using an innocuous
abuse of language, in the following we will simply call $\Vert f\Vert_{m,K}$
a norm. 
\begin{prop}
\noindent \label{prop:genNormForDifferentNets} Under the assumptions
of Def.~\ref{def:generalizedNormsGD_K}, let the set $K=[K_{\eps}]\fcmp\Rtil^{n}$.
Then we have: 
\begin{enumerate}
\item \label{enu:genNormGenNet}If the net $(v_{\eps})$ defines $f$, then
$\Vert f\Vert_{m,K}=\left[\max_{\substack{|\alpha|\le m\\
1\le i\le d
}
}\sup_{x\in K_{\eps}}\left|\partial^{\alpha}v_{\eps}^{i}(x)\right|\right]$ 
\item \label{enu:genNormNetCmptlySupp}If $(u_{\eps})$ defines $f$ and
$(u_{\eps})$ satisfies Def.~\ref{def:cmptlySuppGSF}, then 
\begin{equation}
\Vert f\Vert_{m,K}=\left[\max_{\substack{|\alpha|\le m\\
1\le i\le d
}
}\sup_{x\in\R^{n}}\left|\partial^{\alpha}u_{\eps}^{i}(x)\right|\right].\label{eq:normForGoodNet}
\end{equation}

\end{enumerate}
\end{prop}
\begin{proof}
In proving \ref{enu:genNormGenNet} we will also prove that the norm
$\Vert f\Vert_{m,K}$ is well-defined, i.e.\ it does not depend on
the particular choice of points $m_{\alpha i}$, $M_{\alpha i}$ as
in Def.~\ref{def:generalizedNormsGD_K}. As in the proof of Prop.~\ref{prop:extremeValues},
we get the existence of $\bar{m}_{\alpha i\eps}$, $\bar{M}_{\alpha i\eps}\in K_{\eps}$
such that 
\[
\forall x\in K_{\eps}:\ \partial^{\alpha}v_{\eps}^{i}(\bar{m}_{\alpha i\eps})\le\partial^{\alpha}v_{\eps}^{i}(x)\le\partial^{\alpha}v_{\eps}^{i}(\bar{M}_{\alpha i\eps}).
\]
Hence $\left|\partial^{\alpha}v_{\eps}^{i}(x)\right|\le\max\left(\left|\partial^{\alpha}v_{\eps}^{i}(\bar{m}_{\alpha i\eps})\right|,\left|\partial^{\alpha}v_{\eps}^{i}(\bar{M}_{\alpha i\eps})\right|\right)$.
Thus 
\[
\max_{\substack{|\alpha|\le m\\
1\le i\le d
}
}\sup_{x\in K_{\eps}}\left|\partial^{\alpha}v_{\eps}^{i}(x)\right|\le\max_{\substack{|\alpha|\le m\\
1\le i\le d
}
}\max\left(\left|\partial^{\alpha}v_{\eps}^{i}(\bar{m}_{\alpha i\eps})\right|,\left|\partial^{\alpha}v_{\eps}^{i}(\bar{M}_{\alpha i\eps})\right|\right).
\]
But $\bar{m}_{\alpha i\eps}$, $\bar{M}_{\alpha i\eps}\in K_{\eps}$,
so 
\begin{align*}
\left[\max_{\substack{|\alpha|\le m\\
1\le i\le d
}
}\sup_{x\in K_{\eps}}\left|\partial^{\alpha}v_{\eps}^{i}(x)\right|\right] & =\left[\max_{\substack{|\alpha|\le m\\
1\le i\le d
}
}\max\left(\left|\partial^{\alpha}v_{\eps}^{i}(\bar{m}_{\alpha i\eps})\right|,\left|\partial^{\alpha}v_{\eps}^{i}(\bar{M}_{\alpha i\eps})\right|\right)\right]=\\
 & =\max_{\substack{|\alpha|\le m\\
1\le i\le d
}
}\max\left(\left|\partial^{\alpha}f^{i}(\bar{M}_{\alpha i})\right|,\left|\partial^{\alpha}f^{i}(\bar{m}_{\alpha i})\right|\right).
\end{align*}
From this, both the fact that the norm $\Vert f\Vert_{m,K}$ is well-defined
and claim \ref{enu:genNormGenNet} follow.

\ref{enu:genNormNetCmptlySupp}: By Lemma \ref{lem:supK_eps-RsetminusK_eps},
we have that 
\begin{align*}
\left[\max_{\substack{|\alpha|\le m\\
1\le i\le d
}
}\sup_{x\in K_{\eps}}\left|\partial^{\alpha}u_{\eps}^{i}(x)\right|\right] & =\max_{\substack{|\alpha|\le m\\
1\le i\le d
}
}\left[\sup_{x\in K_{\eps}}\left|\partial^{\alpha}u_{\eps}^{i}(x)\right|\right]=\\
 & =\max_{\substack{|\alpha|\le m\\
1\le i\le d
}
}\left[\sup_{x\in\R^{n}}\left|\partial^{\alpha}u_{\eps}^{i}(x)\right|\right]=\left[\max_{\substack{|\alpha|\le m\\
1\le i\le d
}
}\sup_{x\in\R^{n}}\left|\partial^{\alpha}u_{\eps}^{i}(x)\right|\right].
\end{align*}
\end{proof}
\begin{cor}
\label{cor:normNotDependsOnK}Let $\emptyset\ne K\fcmp U\subseteq\Rtil^{n}$,
where $U$ is a sharply open set. Let $\emptyset\ne H\fcmp U$ and
$m\in\N$. If $f\in\GD_{K}(U,\Rtil^{d})\cap\GD_{H}(U,\Rtil^{d})$,
then $\Vert f\Vert_{m,K}=\Vert f\Vert_{m,H}=:\Vert f\Vert_{m}$.\end{cor}
\begin{proof}
The right hand side of \eqref{eq:normForGoodNet} does not depend
on $K$. 
\end{proof}
Another consequence of Prop.~\ref{prop:genNormForDifferentNets}
is the following: 
\begin{cor}
\label{cor:normOfExtension}Let $U\subseteq\Rtil^{n}$, $f\in\GD(U,\Rtil^{d})$
and $\bar{f}\in\GD(\Rtil^{n},\Rtil^{d})$ be the extension of $f$
defined in Thm.~\ref{thm:extensionCmptlySupportedGSF}. Then for
all $m\in\N$, $\Vert f\Vert_{m}=\Vert\bar{f}\Vert_{m}$. 
\end{cor}
Our use of the term ``norm'' is justified by the following 
\begin{prop}
\label{prop:normProp}Let $\emptyset\ne K\fcmp U\subseteq\Rtil^{n}$,
where $U$ is a sharply open set. Let $f$, $g\in\GD_{K}(U,\Rtil^{d})$
and $m\in\N$. Then 
\begin{enumerate}
\item \label{enu:normProp-pos}$\Vert f\Vert_{m}\ge0$ 
\item \label{enu:normProp-zero}$\Vert f\Vert_{m}=0$ if and only if $f=0$ 
\item \label{enu:normProp-scalar}$\forall c\in\Rtil:\ \Vert c\cdot f\Vert_{m}=|c|\cdot\Vert f\Vert_{m}$ 
\item \label{enu:normProp-triangle}$\Vert f+g\Vert_{m}\le\Vert f\Vert_{m}+\Vert g\Vert_{m}$. 
\item \label{enu:normProp-prod}$\Vert f\cdot g\Vert_{m}\le c_{m}\cdot\Vert f\Vert_{m}\cdot\Vert g\Vert_{m}$
for some $c_{m}\in\R_{>0}$. 
\end{enumerate}
\end{prop}
\begin{proof}
\ref{enu:normProp-pos}, \ref{enu:normProp-scalar} and \ref{enu:normProp-triangle}
follow directly from Prop.~\ref{prop:genNormForDifferentNets}, as
does \ref{enu:normProp-prod}, using the Leibniz rule. The `if'-part
of property \ref{enu:normProp-zero} follows from \eqref{eq:normForGoodNet}. 
\end{proof}
We now prove that also the space $\GD(U,\Rtil^{d})$ is an $\Rtil$-module,
at least for certain $U$: 
\begin{prop}
\label{prop:GD-Rtil-Module}Let $U\subseteq\Rtil^{n}$ be a non empty
sharply open set. Assume that 
\begin{equation}
\forall K,H\fcmp U:\ \text{\emph{interl}}(H\cup K)\subseteq U.\label{eq:suffCondGDRtilModule}
\end{equation}
Then $\GD(U,\Rtil^{d})$ is an $\Rtil$-module.\end{prop}
\begin{proof}
Since in Thm.~\ref{thm:GD_K-Rtil-Module} we already proved that
$\GD_{K}(U,\Rtil^{d})$ is closed with respect to products by scalars,
we only need to prove that $\GD(U,\Rtil^{d})$ is closed with respect
to sum. Let $f\in\GD_{K}(U,\Rtil^{d})$, $g\in\GD_{H}(U,\Rtil^{d})$
and let $(u_{\eps})$, $(v_{\eps})$ satisfy Def.~\ref{def:cmptlySuppGSF}
for $f$ and $g$, respectively. Lemma \ref{thm:unionIntersection}
and \cite[Prop. 2.8]{ObVe08} imply that $\text{interl}(H\cup K)=[H_{\eps}\cup K_{\eps}]\fcmp\Rtil^{n}$.
But $\exterior{H\cup K}=\sint{H_{\eps}^{c}\cap K_{\eps}^{c}}=\sint{H_{\eps}^{c}}\cap\sint{K_{\eps}^{c}}=\exterior{H}\cap\exterior{K}$.
Therefore, $\partial^{\alpha}(u_{\eps}+v_{\eps})$ is zero on $\exterior{H\cup K}$.
By our assumption \eqref{eq:suffCondGDRtilModule} $\emptyset\ne\text{interl}(H\cup K)\fcmp U$,
so that $f+g\in\GD(U,\Rtil^{d})$. 
\end{proof}
In the following result, we give two general sufficient conditions
for \eqref{eq:suffCondGDRtilModule} to hold. 
\begin{prop}
\label{prop:suffCondsInterlUnionIn-U}Let $U\subseteq\Rtil^{n}$ be
a non-empty sharply open set. If $U$ is $\Rtil$-convex or $U$ is
a strongly internal set, then \eqref{eq:suffCondGDRtilModule} holds.\end{prop}
\begin{proof}
Assume that $U$ is $\Rtil$-convex, i.e.\ $xh+(1-x)k\in U$ for
all $h$, $k\in U$ and all $x\in[0,1]$. Then for all $H$, $K\subseteq U$
(even if we do not assume them to be functionally compact), and all
$y\in\text{interl}(H\cup K)$, we can write $y=e_{S}\cdot h+e_{S^{c}}\cdot k$
for some $S\subseteq I$ and $h\in H$, $k\in K$. Thus $y=e_{S}\cdot h+(1-e_{S})\cdot k\in U$
since $e_{S}\in[0,1]$ and $h$, $k\in U$.

Now, assume that $U$ is a strongly internal set, i.e.\ for some
net $(U_{\eps})$ of subsets of $\R^{n}$, we have $U=\sint{U_{\eps}}$.
We continue to use the notations for $y$ as above. Since $h$, $k\in U$,
\cite[Thm. 8]{GKV} entails that $d(h_{\eps},U_{\eps}^{c})$, $d(k_{\eps},U_{\eps}^{c})>\eps^{q}$
for some $q\in\R_{>0}$ and $\eps$ small, where $h=[h_{\eps}]$,
$k=[k_{\eps}]$. But $y=e_{S}\cdot h+e_{S^{c}}\cdot k$, so for all
$\eps$ small, if $\eps\in S$ then $y_{\eps}=h_{\eps}$ and if $\eps\notin S$
then $y_{\eps}=k_{\eps}$. In any case, $d(y_{\eps},U_{\eps}^{c})>\eps^{q}$,
hence $y\in\sint{U_{\eps}}=U$.\end{proof}
\begin{example}
\label{exa:UDoesNotContainInterl}If $U=(-1,1)\cup(2,4)\subseteq\Rtil$,
then $U$ is a sharply open set, but it does not satisfy condition
\eqref{eq:suffCondGDRtilModule} of Prop.~\eqref{prop:GD-Rtil-Module}:
let $H:=\left[\left[-\frac{1}{2},\frac{1}{2}\right]_{\R}\right]$,
$K:=\left[\left[\frac{5}{2},\frac{7}{2}\right]_{\R}\right]$ and $x_{\eps}:=0$
if $\eps\in\left[\left.\frac{1}{n},\frac{1}{n+1}\right)\right.$ if
$n\in\N_{>0}$ is even, and $x_{\eps}:=3$ otherwise. Then $x:=[x_{\eps}]\in\text{interl}(H\cup K)$
but $x\notin U$. Moreover, let $\phi\in\D_{\left[-\frac{1}{2},\frac{1}{2}\right]_{\R}}\left(\R\right)\subseteq\GD_{H}(\Rtil_{c})$,
$\psi\in\D_{\left[\frac{5}{2},\frac{7}{2}\right]_{\R}}(\R)\subseteq\GD_{K}(\Rtil_{c})$
be positive non-trivial functions. Then, as we showed in the proof
of Prop.~\ref{prop:GD-Rtil-Module}, the GSF $\phi+\psi\in\Gcinf(U,\Rtil)$
is compactly supported in $\left[\left[-\frac{1}{2},\frac{1}{2}\right]_{\R}\cup\left[\frac{5}{2},\frac{7}{2}\right]_{\R}\right]=\text{interl}(H\cup K)\not\subseteq U$.
Finally, let $f:=\phi|_{U}$ and $g:=\psi|_{U}$, so that $f\in\GD_{H}(U,\Rtil)$
and $g\in\GD_{K}(U,\Rtil)$. Rem.~\ref{rem:defCmptlySuppGSF}.\ref{enu:suppSubsetK}
yields that $f+g\notin\GD_{J}(U,\Rtil)$ for all $J\fcmp U$: Otherwise,
taking suitable sub-intervals $L$ of $\left(-\frac{1}{2},\frac{1}{2}\right)_{\R}$
and $M$ of $\left(\frac{5}{2},\frac{7}{2}\right)_{\R}$ where $f$
resp.~$g$ do not vanish, we would have $[L\cup M]\subseteq\text{supp}(f+g)\subseteq J\subseteq U$
(here $\text{supp}(f+g)$ is the support as in Def.~\ref{def:support}).
But the inclusion $[L\cup M]\subseteq U$ is impossible --- a counterexample
can be constructed similar to the above $x$.

\noindent This example shows that an assumption like \eqref{eq:suffCondGDRtilModule}
is necessary to have the closure of the space $\GD(U,\Rtil)$ with
respect to sum. 
\end{example}

\section{\label{sec5}Topological structure on \texorpdfstring{$\GD_{K}$}{GDK}}

Using our $\Rtil$-valued norms, it is now natural to define 
\begin{defn}
\label{def:ballsTopology}Let $\emptyset\ne K\fcmp U\subseteq\Rtil^{n}$,
where $U$ is a sharply open set. Let $f\in\GD_{K}(U,\Rtil^{d})$,
$m\in\N$, $\rho\in\Rtil_{>0}$, then 
\begin{enumerate}
\item \label{enu:defBall} $B_{\rho}^{m}(f):=\left\{ g\in\GD_{K}(U,\Rtil^{d})\mid\Vert f-g\Vert_{m}<\rho\right\} $.
In case any confusion could arise, we will use the more precise symbol
$B_{\rho}^{m}(f,K):=B_{\rho}^{m}(f)$. 
\item \label{enu:sharpOpenInGD_K}If $V\subseteq\GD_{K}(U,\Rtil^{d})$,
then we say that $V$ is a \emph{sharply open set} if 
\[
\forall v\in V\,\exists m\in\N\,\exists\rho\in\Rtil_{>0}:\ B_{\rho}^{m}(v)\subseteq V.
\]
Moreover, we say that $V$ is \emph{Fermat open} if 
\[
\forall v\in V\,\exists m\in\N\,\exists r\in\R_{>0}:\ B_{r}^{m}(v)\subseteq V.
\]

\end{enumerate}
\end{defn}
As in \cite[Thm. 2]{GKV} it follows that sharply open sets as well
as Fermat open sets form topologies on $\GD_{K}(U,\Rtil^{d})$.

\noindent On the other hand, it is also natural to view the space
$\GD_{K}(U,\Rtil^{d})$ inside Garetto's theory \cite{Gar05,Gar05b}
of $\Rtil$-locally convex algebras. In this section, we will realize
this comparison, proving that the space $\GD_{K}(U,\Rtil)$ is a Fréchet
$\Rtil$-module and a topological algebra. For this purpose, we will
only consider the sharp topology. Indeed, as we will see below, the
Fermat topology is less interesting in this context since it doesn't
permit to prove the continuity of the product by scalars $(r,f)\in\Rtil\times\GD_{K}(U,\Rtil^{d})\mapsto r\cdot f\in\GD_{K}(U,\Rtil^{d})$.

\noindent In the following, we will always assume that $\emptyset\ne K\fcmp U\subseteq\Rtil^{n}$,
where $U$ is a non-empty sharply open set. The main problem in performing
this comparison, which doesn't permit to view our space $\GD_{K}(U,\Rtil)$
as a particular case of the theory developed in \cite{Gar05,Gar05b},
is that the domain $U$ contains generalized points.

Using the valuation $v$ on $\Rtil$, it is natural to introduce the
following notions: 
\begin{defn}
\label{def:valuationPseudoUltraNorm}Let $m\in\N$ and $f\in\GD(U,\Rtil)$,
then: 
\begin{enumerate}
\item \label{enu:valuation-GD}$v_{m}(f):=v(\Vert f\Vert_{m})\in\R$ 
\item \label{enu:pseudoUltraNorm}$\mathcal{P}_{m}(f):=e^{-v_{m}(f)}.$ 
\end{enumerate}
\end{defn}
From the properties of the valuation $v$ and of the e-norm $|-|_{e}=e^{-v(-)}$
on $\Rtil$ (see \cite{AJ}), the following result directly follows. 
\begin{prop}
\label{prop:propValuationsUPnorm}For each $m\in\N$, we have: 
\begin{enumerate}
\item \label{enu:valuationGD}$v_{m}:\GD(U)\ra\R\cup\{+\infty\}$ is a valuation,
i.e.\ for all $f$, $g\in\GD(U)$: 
\begin{itemize}
\item $v_{m}(0)=+\infty$ 
\item $v_{m}(\lambda\cdot f)\ge v(\lambda)+v_{m}(f)\quad\forall\lambda\in\Rtil$ 
\item $v_{m}(\diff{\eps}^{a}\cdot f)=v(\diff{\eps}^{a})+v_{m}(f)=a+v_{m}(f)\quad\forall a\in\R$ 
\item $v_{m}(f+g)\ge\min\left[v_{m}(f),v_{m}(g)\right]$. 
\end{itemize}
\item \label{enu:UPnormGD}$\mathcal{P}_{m}:\GD(U)\ra\R$ is an ultra-pseudo-norm,
i.e.\ for all $f$, $g\in\GD(U)$:
\begin{itemize}
\item $\mathcal{P}_{m}(f)=0$ if and only if $f=0$ 
\item $\mathcal{P}_{m}(\lambda\cdot f)\le|\lambda|_{e}\cdot\mathcal{P}_{m}(f)\quad\forall\lambda\in\Rtil$ 
\item $\mathcal{P}_{m}(\diff{\eps}^{a}\cdot f)=|\diff{\eps}^{a}|_{e}\cdot\mathcal{P}_{m}(f)=e^{-a}\cdot\mathcal{P}_{m}(f)\quad\forall a\in\R$ 
\item $\mathcal{P}_{m}(f+g)\le\max\left[\mathcal{P}_{m}(f),\mathcal{P}_{m}(g)\right]$. 
\end{itemize}
\end{enumerate}
\end{prop}
The following result states that to define the sharp topology, instead
of the above ultra-pseudo-norms we can equivalently use the countable
family of generalized norms $\left(\Vert f\Vert_{m}\right)_{m\in\N}$. 
\begin{thm}
\label{thm:sharpTop}\ 
\begin{enumerate}
\item \label{enu:RtilAlgebraSharp}Sum and product in $\GD_{K}(U)$ are
continuous in the sharp topology. Therefore, $\GD_{K}(U)$ is a topological
$\Rtil$-algebra. 
\item \label{enu:productAndFermat}The product in $\GD_{K}(U)$ is continuous
in the Fermat topology only on the subspace $\{(f,g)\in\GD_{K}(U)\times\GD_{K}(U)\mid\forall\,m\in\N:\ \Vert f\Vert_{m},\ \Vert g\Vert_{m}<\infty\}$. 
\item \label{enu:ballsWithUPnorms-GenNorm}For $f\in\GD_{K}(U)$, set 
\[
C_{r}^{m}(f):=\left\{ g\in\GD_{K}(U)\mid\mathcal{P}_{m}(f-g)<r\right\} \quad(r\in\R_{>0},\ m\in\N).
\]
Then for each $q$, $s\in\R_{>0}$ we have:

\begin{enumerate}
\item If $q\le-\log r$, then $C_{r}^{m}(f)\subseteq B_{\text{\emph{d}}\eps^{q}}^{m}(f)$ 
\item If $q\ge-\log s$ and $s<r$, then $B_{\text{\emph{d}}\eps^{q}}^{m}(f)\subseteq C_{r}^{m}(f)$. 
\end{enumerate}
\item \label{enu:SharpCoarsest-P_m}The sharp topology on $\GD_{K}(U)$
is the coarsest topology such that each $\mathcal{P}_{m}$ is continuous. 
\item \label{enu:LCTSep}$\GD_{K}(U)$ is a separated locally convex topological
$\Rtil$-module. 
\end{enumerate}
\end{thm}
\begin{proof}
\ref{enu:RtilAlgebraSharp}, \ref{enu:productAndFermat}: Continuity
of the sum in the sharp (and in the Fermat) topology follows directly
from the triangle inequality Prop.~\ref{prop:normProp}.\ref{enu:normProp-triangle}.
The continuity of the product at $(f_{0},g_{0})$ and property \ref{enu:productAndFermat}
follow from Prop.~\ref{prop:normProp} \ref{enu:normProp-prod} via
\begin{align*}
\Vert f\cdot g-f_{0}\cdot g_{0}\Vert_{m} & \le c_{m}(\Vert f-f_{0}\Vert_{m}\cdot\Vert g-g_{0}\Vert_{m}+\Vert f-f_{0}\Vert_{m}\cdot\Vert g_{0}\Vert_{m}+\\
 & \phantom{\le}+\Vert f_{0}\Vert_{m}\cdot\Vert g-g_{0}\Vert_{m}).
\end{align*}

\ref{enu:ballsWithUPnorms-GenNorm}: Let us first assume $q\le-\log r$
and $g\in C_{r}^{m}(f)$, so that $\mathcal{P}_{m}(f-g)<r$ and $v_{m}(f-g)>-\log r$.
This implies 
\begin{equation}
\exists b>-\log r:\ \max_{\substack{|\alpha|\le m\\
i\le d
}
}\sup_{x\in K_{\eps}}\left|\partial^{\alpha}u_{\eps}^{i}(x)-\partial^{\alpha}v_{\eps}^{i}(x)\right|=O(\eps^{b}),\label{eq:maxSupBigO}
\end{equation}
where $(u_{\eps})$ and $(v_{\eps})$ define $f$ and $g$, respectively.
Property \eqref{eq:maxSupBigO} yields the existence of some $M>0$
such that for $\eps$ sufficiently small we obtain 
\[
\ \max_{\substack{|\alpha|\le m\\
i\le d
}
}\sup_{x\in K_{\eps}}\left|\partial^{\alpha}u_{\eps}^{i}(x)-\partial^{\alpha}v_{\eps}^{i}(x)\right|+\eps^{b}\le(M+1)\cdot\eps^{b}<\eps^{-\log r}\le\eps^{q}.
\]
Therefore $\Vert f-g\Vert_{m}+\diff{\eps^{b}}\le\diff{\eps}^{q}$,
so $\Vert f-g\Vert_{m}<\diff{\eps}^{q}$.

Now, let us assume $q\ge-\log s$, $s<r$, and $g\in B_{\text{d}\eps^{q}}^{m}(f)$,
so that $\Vert f-g\Vert_{m}<\diff{\eps}^{q}$. Therefore, 
\[
\forall^{0}\eps:\ \max_{\substack{|\alpha|\le m\\
i\le d
}
}\sup_{x\in K_{\eps}}\left|\partial^{\alpha}u_{\eps}^{i}(x)-\partial^{\alpha}v_{\eps}^{i}(x)\right|<\eps^{q}\le\eps^{-\log s}.
\]
Taking the $|-|_{e}$-norm we get 
\[
\left|\max_{\substack{|\alpha|\le m\\
i\le d
}
}\sup_{x\in K_{\eps}}\left|\partial^{\alpha}u_{\eps}^{i}(x)-\partial^{\alpha}v_{\eps}^{i}(x)\right|\right|_{e}\le e^{\log s}=s<r,
\]
that is $\mathcal{P}_{m}(f-g)<r$ as claimed.

\ref{enu:SharpCoarsest-P_m} follows directly from \ref{enu:ballsWithUPnorms-GenNorm}
and Prop.~\ref{prop:propValuationsUPnorm}.\ref{enu:UPnormGD}.

\ref{enu:LCTSep} follows from Prop.~\ref{prop:normProp}.\ref{enu:normProp-zero}
and \cite[Prop. 1.11]{Gar05}. 
\end{proof}

\subsection{\label{sub:5.1}Generalized functions and non-Archimedean properties}

In this section, we want to clarify some relationships between the
classical notion of convexity, the notion of $\Rtil$-convexity of
\cite{Gar05} and the use of $\Rtil$-valued norms.

We have seen that balls $B_{\rho}^{m}(0)$, $\rho\in\Rtil_{>0}$,
define a neighborhood system of $0$ for $\GD_{K}(U)$; they are convex
in the usual sense, i.e.\ if $f$, $g\in B_{\rho}^{m}(0)$, $t\in[0,1]$
(in particular if $t\in[0,1]_{\R})$, then 
\[
\Vert tf+(1-t)g\Vert_{m}\le t\Vert f\Vert_{m}+(1-t)\Vert g\Vert_{m}<t\rho+(1-t)\rho=\rho.
\]
Moreover, each ball $B_{\rho}^{m}(0)$ is also balanced: if $t\in\Rtil$,
$|t|\le1$, then $t\cdot B_{\rho}^{m}(0)\subseteq B_{\rho}^{m}(0)$.
However, this space is not a classical locally convex topological
vector space over the field $\R$ because of two reasons: (i) the
product by scalars is not continuous with respect to the Euclidean
topology on $\R$, (ii) Lemma \ref{lem:cmptlySuppGSFareBounded} implies
that the property 
\[
\forall f\in\GD_{K}(U)\,\exists t\in\Rtil:\ f\in t\cdot B_{\rho}^{m}(0)
\]
holds for $t\in\Rtil$ but it cannot be extended to $t\in\R$. As
we have seen in the proof of Thm.~\ref{thm:sharpTop} \ref{enu:productAndFermat},
this is a necessary consequence of the existence of generalized functions
with infinite $\Rtil$-valued $\Vert-\Vert_{m}$-norm.

On the other hand, even though the sets $C_{r}^{m}(0)$ are defined
using $\R$ only, i.e.\ without mentioning any non-Archimedean property,
they satisfy 
\begin{equation}
\forall f\in C_{r}^{m}(0)\,\forall\lambda\in\R:\ \lambda\cdot f\in C_{r}^{m}(0),\label{eq:homog}
\end{equation}
and this is possible only because they are infinitesimal sets. In
fact, we have seen in Thm.~\ref{thm:sharpTop} that $C_{r}^{m}(0)\subseteq B_{\text{d}\eps^{q}}^{m}(0)$
for $q\le-\log r$.

More generally, a set $A\subseteq\GD_{K}(U)$ can be $\Rtil$-balanced
(see \cite{Gar05}), i.e. 
\[
\lambda A\subseteq A\quad\forall\lambda\in\Rtil:\ |\lambda|_{e}\le1,
\]
and at the same time can be thought of as ``small'' only in case
$A$ consists infinitesimal points. For example, the ball $B_{\text{d}\eps^{b}}^{m}(0)$
is $\Rtil$-balanced, but $B_{1}^{m}(0)$ is not. In fact, we have 
\begin{lem}
\label{lem:infinitesimalA} Suppose that $A\subseteq\GD_{K}(U)$ and
$m\in\N$ are such that 
\begin{equation}
A+A\subseteq A\label{eq:APlusAInA}
\end{equation}
\[
\exists r\in\R_{>0}:\ A\subseteq B_{r}^{m}(0).
\]
Then every element $u\in A$ has infinitesimal norm: $\Vert u\Vert_{m}\approx0$.\end{lem}
\begin{proof}
In fact, \eqref{eq:APlusAInA} implies $n\cdot u\in A\subseteq B_{r}^{m}(0)$
for all $n\in\N_{\ne0}$. Therefore, $\Vert u\Vert_{m}<\frac{r}{n}$
for all $n\in\N_{\ne0}$, which proves our claim. 
\end{proof}
Let us note that condition \eqref{eq:APlusAInA} holds both for $A$
which is $\Rtil$-balanced or $\Rtil$-convex.

\noindent These remarks permit to show that in dealing with generalized
functions, we are naturally induced to consider a topology on $\Rtil$
which contains infinitesimal neighborhoods (hence inducing the discrete
topology on $\R$, see \cite{GK2}). This is due to the coexistence
of a continuous product by scalars and of an infinite element in $\GD_{K}(U)$,
as stated in the following general result. In a possible interpretation
of its statement, we can think of $R$ as $\R$ with a topology $\tau$,
$\widetilde{R}$ as $\Rtil$ with the sharp topology $\tilde{\tau}$,
and $<_{{\scriptscriptstyle \widetilde{R}}}$ as the strict order
relation $<$ of Lemma \ref{lem:mayer}. 
\begin{thm}
\noindent \label{thm:discreteTopOnR-General}Let $(R,+_{{\scriptscriptstyle R}},\cdot_{{\scriptscriptstyle R}},<_{{\scriptscriptstyle R}},\tau)$
and $(\widetilde{R},+_{{\scriptscriptstyle \widetilde{R}}},\cdot_{{\scriptscriptstyle \widetilde{R}}},<_{{\scriptscriptstyle \widetilde{R}}},\tilde{\tau})$
be Hausdorff topological ordered rings such that $(\R,+,\cdot,<)$
is a substructure of $(R,+_{{\scriptscriptstyle R}},\cdot_{{\scriptscriptstyle R}},<_{{\scriptscriptstyle R}})$,
which in turn is a substructure of $(\widetilde{R},+_{{\scriptscriptstyle \widetilde{R}}},\cdot_{{\scriptscriptstyle \widetilde{R}}},<_{{\scriptscriptstyle \widetilde{R}}})$
and such that 
\begin{equation}
\forall r\in R\ \forall s\in\R:\ r<_{\widetilde{{\scriptscriptstyle R}}}s\then r<_{{\scriptscriptstyle R}}s.\label{eq:strictPresLess}
\end{equation}
Let $(G,+_{{\scriptscriptstyle G}},\cdot{\scriptscriptstyle G},\sigma)$
be a Hausdorff topological $R$-module, and $|-|_{{\scriptscriptstyle G}}:G\ra\widetilde{R}$,
$|-|_{{\scriptscriptstyle R}}:R\ra R$ be maps such that $\left|r\cdot_{{\scriptscriptstyle G}}g\right|_{{\scriptscriptstyle G}}=|r|_{{\scriptscriptstyle R}}\cdot_{{\scriptscriptstyle \widetilde{R}}}|g|_{{\scriptscriptstyle G}}$
for all $r\in R$ and all $g\in G$. Assume that any $\tau$-neighborhood
of $0\in R$ contains a ball $B_{\eta}^{R}(0):=\left\{ s\in R\mid|s|_{{\scriptscriptstyle R}}<_{{\scriptscriptstyle R}}\eta\right\} $
for some $\eta\in R$, $\eta>_{{\scriptscriptstyle R}}0$, and that
there exists some $\rho\in\widetilde{R}$ with $\rho>_{{\scriptscriptstyle \widetilde{R}}}0$
such that the ball $B_{\rho}^{G}(0):=\left\{ g\in G\mid|g|_{{\scriptscriptstyle G}}<_{{\scriptscriptstyle \widetilde{R}}}\rho\right\} $
is $\sigma$-open. Finally, assume that 
\begin{equation}
\exists g\in G:\ |g|_{{\scriptscriptstyle G}}\text{ is invertible}\ ,\ \forall M\in\R_{>0}:\ |g|_{{\scriptscriptstyle G}}>_{{\scriptscriptstyle \widetilde{R}}}\rho\cdot_{{\scriptscriptstyle \widetilde{R}}}M.\label{eq:invertibleElementG}
\end{equation}
Then the induced topology $\tau\cap\R$ is discrete.\end{thm}
\begin{proof}
\noindent Since $G$ is a topological $R$-module, the product by
scalars is $\tau\times\sigma$-continuous, and 
\begin{equation}
\lim_{\substack{r\to0\\
r\in R
}
}r\cdot_{{\scriptscriptstyle G}}g=0,\label{eq:limit-r-g}
\end{equation}
where $g\in G$ comes from assumption \eqref{eq:invertibleElementG}.
By hypothesis, the ball $B_{\rho}^{G}(0)\in\sigma$ and every $\tau$-neighborhood
of $r=0\in R$ contains some ball $B_{\eta}^{R}(0)$. Therefore \eqref{eq:limit-r-g}
entails that there exists some $\eta>_{{\scriptscriptstyle R}}0$
such that 
\begin{equation}
\forall r\in R:\ |r|_{{\scriptscriptstyle R}}<_{{\scriptscriptstyle R}}\eta\then\rho>_{{\scriptscriptstyle \widetilde{R}}}|r\cdot_{{\scriptscriptstyle G}}g|_{{\scriptscriptstyle G}}=|r|_{{\scriptscriptstyle R}}\cdot_{{\scriptscriptstyle \widetilde{R}}}|g|_{{\scriptscriptstyle G}}.\label{eq:limitExplicit}
\end{equation}
For each $s\in\R_{>0}$ take $M\in\R_{>0}$ such that $\frac{1}{M}<s$,
so that $|g|_{{\scriptscriptstyle G}}>_{{\scriptscriptstyle \widetilde{R}}}\rho\cdot_{{\scriptscriptstyle \widetilde{R}}}M>_{{\scriptscriptstyle \widetilde{R}}}0$
from \eqref{eq:invertibleElementG}. For all $r\in R$ such that $|r|_{{\scriptscriptstyle R}}<_{{\scriptscriptstyle R}}\eta$,
we have 
\[
|r|_{{\scriptscriptstyle R}}<_{{\scriptscriptstyle \widetilde{R}}}\frac{\rho}{|g|_{{\scriptscriptstyle G}}}<_{{\scriptscriptstyle \widetilde{R}}}\frac{1}{M}<s
\]
because $|g|_{{\scriptscriptstyle G}}$ is invertible in $\widetilde{R}$.
Therefore $|r|_{{\scriptscriptstyle R}}<_{{\scriptscriptstyle \widetilde{R}}}s$
and hence $|r|_{{\scriptscriptstyle R}}<_{{\scriptscriptstyle R}}s$
by \eqref{eq:strictPresLess}. This means that $r$ is infinitesimal
in the ring $R$, i.e.\ the ball $B_{\rho}^{R}(0)$ is contained
in the monad of $0$ (see e.g.~\cite{GK2} for the notion of monad)
and so also every ball $B_{\eta}^{R}(\bar{r})$ is contained in the
monad of $\bar{r}\in R$. Therefore, \cite[Prop. 2.1]{GK2} implies
the conclusion. 
\end{proof}
\noindent We can therefore say that if we want to find a space $G$
of generalized functions which is an ordinary Hausdorff topological
vector space on $\R$, then we cannot define the topologies $\tau$
and $\sigma$ using seminorms valued in a non-Archimedean (see \eqref{eq:invertibleElementG})
extension of $\R$. This results confirms Rem.\ 43 of \cite{DeHaPiVa07}.

\noindent As a consequence, we have the following impossibility result:
\begin{cor}
\noindent \label{cor:impLCTVS}There does not exist any real Hausdorff
topological vector space $(G,+_{{\scriptscriptstyle G}},\cdot{\scriptscriptstyle G},\sigma)$
such that: 
\begin{enumerate}
\item \label{enu:G-SubsetGD_K}$(G,+_{{\scriptscriptstyle G}},\cdot{\scriptscriptstyle G})$
is a linear subspace of $\GD_{K}(U)$ for some $\emptyset\ne K\fcmp U\subseteq\Rtil^{n}$ 
\item \label{enu:deltaIn-G}$G$ contains the Dirac delta $\delta\in G$ 
\item \label{enu:ballInSigma}$\exists m\in\N\,\exists\rho\in\Rtil_{>0}:\rho<1\ ,\ B_{\rho}^{m}(0)\cap G\in\sigma$ 
\end{enumerate}
\noindent In particular, the Colombeau algebra $\gs(\Omega)$ does
not contain any real Hausdorff topological vector subspace $G$ such
that some $B_{\rho}^{m}(0)\cap G$ is open and $\delta\in G$.\end{cor}
\begin{proof}
\noindent By contradiction, in Thm.~\ref{thm:discreteTopOnR-General},
set $R:=\R$ with the usual Euclidean topology $\tau$, $\widetilde{R}:=\Rtil$
with the sharp topology; set $|g|_{{\scriptscriptstyle G}}:=\Vert g\Vert_{m}$,
where $m\in\N$ comes from \ref{enu:ballInSigma} and we used the
inclusion \ref{enu:G-SubsetGD_K}; set $|r|_{{\scriptscriptstyle R}}:=|r|$
the usual absolute value in $\R$. Note also that $B_{\rho}^{m}(0)\cap G=\{g\in G\mid\Vert g\Vert_{m}<\rho\}=\{g\in G\mid|g|_{G}<_{{\scriptscriptstyle \widetilde{R}}}\rho\}$.
If $\delta\in G$, then $|\delta|_{{\scriptscriptstyle G}}=\Vert\delta\Vert_{m}$
is infinite and invertible in $\Rtil$, so that Thm.~\ref{thm:discreteTopOnR-General}
implies that the Euclidean topology would be discrete. The second
part of the claim follows from Thm.~\ref{thm:CGFasCompctlySuppGSF}. 
\end{proof}
We can summarize Cor.~\ref{cor:impLCTVS} by saying that a real Hausdorff
topological vector space structure $G$ for a space of generalized
functions cannot contain even a single trace $B_{\rho}^{m}(0)\cap G$,
$\rho<1$, of a sharply open ball. This result does not contradict
\cite[Prop. 4]{ACJ}, where it is stated that the sharp topology induces
on \emph{bounded} sets of the real locally convex space $\mathcal{G}_{a}(\Omega)\subseteq\gs(\Omega)$
a topology which is \emph{finer} than the topology $\sigma_{a}$ on
$\mathcal{G}_{a}(\Omega)$. On the other hand, Cor.~\ref{cor:impLCTVS}
implies that $B_{\rho}^{m}(0)\cap\mathcal{G}_{a}(\Omega)\notin\sigma_{a}$
for all $\rho\in\Rtil_{>0}$, $\rho<1$.

\section{\label{sec6}Metric structure on \texorpdfstring{$\GD_{K}$}{GDK}}

In this section, we want to use \cite[Thm. 1.14]{Gar05} to prove
metrizability of $\GD_{K}(U)$. However, we will apply this result
using an explicit and simple countable base of neighborhoods of the
origin which consists of $\Rtil$-absorbent and absolutely $\Rtil$-convex
sets. In this way, we will arrive at a simpler metric.

The idea is to consider only points of balls $f\in B_{\rho}^{m}(0)$
whose norm $\Vert f\Vert_{m}$ is infinitely smaller than $\rho$.
To formally express this idea, we introduce the following 
\begin{defn}
\label{def:U-rho-m}Let $\rho\in\Rtil_{>0}$, $m\in\N$ and $g\in\GD_{K}(U)$.
Then 
\[
U_{\rho}^{m}(g):=\left\{ f\in\GD_{K}(U)\mid\frac{1}{\rho}\cdot\Vert f-g\Vert_{m}\approx0\right\} .
\]
In case any confusion might arise, we will use the more precise symbol
$U_{\rho}^{m}(g,K):=U_{\rho}^{m}(g)$.\end{defn}
\begin{prop}
\label{prop:U-rho-m-topology}\ 
\begin{enumerate}
\item \label{enu:U-rho-m-abs}$U_{\rho}^{m}(0)$ is $\Rtil$-absorbent and
absolutely $\Rtil$-convex. 
\item \label{enu:U-rho-m-sharpTop} Both the system 
\[
\left\{ U_{\rho}^{m}(v)\mid v\in\GD_{K}(U),\ m\in\N,\ \rho\in\Rtil_{>0},\ \rho\approx0\right\} ,
\]
and the system 
\[
\left\{ U_{\text{\emph{d}}\eps^{n}}^{n}(v)\mid v\in\GD_{K}(U),\ n\in\N_{>0}\right\} 
\]
generate the sharp topology on $\GD_{K}(U)$. 
\end{enumerate}
\end{prop}
\begin{proof}
To see that $U_{\rho}^{m}(0)$ is $\Rtil$-absorbent, by \cite[Def. 1.1]{Gar05}
we have to show that 
\begin{equation}
\forall u\in\GD_{K}(U)\,\exists a\in\R\,\forall b\in\R_{\le a}:\ u\in\diff{\eps}^{b}\cdot U_{\rho}^{m}(0),\label{eq:U-rho-m-Absorbent}
\end{equation}
i.e.\ $\frac{u}{\text{d}\eps^{b}}\in U_{\rho}^{m}(0)$, which is
equivalent to $\left\Vert \frac{u}{\text{d}\eps^{b}}\right\Vert _{m}\cdot\frac{1}{\rho}=\frac{\Vert u\Vert_{m}}{\text{d}\eps^{b}\cdot\rho}\approx0$.
But $\rho$ is strictly positive, so $\rho\ge\diff{\eps}^{p}$ for
some $p\in\R$. Moreover, $\Vert u\Vert_{m}\in\Rtil$, so $\Vert u\Vert_{m}\le\diff{\eps}^{q}$
for some $q\in\R$. Therefore 
\[
\frac{\Vert u\Vert_{m}}{\diff{\eps}^{b}\cdot\rho}\le\diff{\eps}^{q-b-p},
\]
and we have $\diff{\eps}^{q-b-p}\approx0$ if and only if $b<q-p$.
This proves \eqref{eq:U-rho-m-Absorbent}.

To prove that $U_{\rho}^{m}(0)$ is balanced, assume $\lambda\in\Rtil$
with $|\lambda|_{e}\le1$. Then $v(\lambda)\ge0$, so $|\lambda|\le c$
for some $c\in\R_{>0}$. Therefore, if $u\in U_{\rho}^{m}(0)$ then
\[
\frac{\Vert\lambda u\Vert_{m}}{\rho}=|\lambda|\frac{\Vert u\Vert_{m}}{\rho}\approx0,
\]
so $\lambda u\in U_{\rho}^{m}(0)$.

Finally, we show $\Rtil$-convexity: for all $a$, $b\in\R_{\ge0}$
and all $u$, $v\in U_{\rho}^{m}(0)$, we have 
\[
\frac{1}{\rho}\cdot\Vert\diff{\eps}^{a}\cdot u+\diff{\eps}^{b}\cdot v\Vert_{m}\le\diff{\eps}^{a}\cdot\frac{\Vert u\Vert_{m}}{\rho}+\diff{\eps}^{b}\cdot\frac{\Vert v\Vert_{m}}{\rho}\approx0.
\]

In order to prove \ref{enu:U-rho-m-sharpTop}, we note that $U_{\rho}^{m}(v)\subseteq B_{\rho}^{m}(v)$
because $\frac{\Vert u-v\Vert_{m}}{\rho}\approx0$ implies $\frac{\Vert u-v\Vert_{m}}{\rho}<1$.
Also, if $\rho\approx0$, then $B_{\rho}^{m}(v)\subseteq U_{\sqrt{\rho}}^{m}(v)$
because $\Vert u-v\Vert_{m}<\rho$ implies $\frac{1}{\sqrt{\rho}}\cdot\Vert u-v\Vert_{m}\le\sqrt{\rho}\approx0$.
Finally, if $\rho\approx0$, every $U_{\rho}^{m}(v)$ is sharply open:
if $u\in U_{\rho}^{m}(v)$ and $w\in B_{\rho^{2}}^{m}(u)$, then $\frac{1}{\rho}\cdot\Vert w-v\Vert_{m}\le\frac{1}{\rho}\cdot\Vert w-u\Vert_{m}+\frac{1}{\rho}\cdot\Vert u-v\Vert_{m}\approx0$.
The proof for the second system in \ref{enu:U-rho-m-sharpTop} follows
by observing that given $\rho>0$, $\rho\approx0$, there exists $q\in\N$
such that $\rho\ge\text{d}\eps^{q}$, and setting $n:=\max(m,q)$
we have $U_{\text{d}\eps^{n}}^{n}(v)\subseteq U_{\rho}^{m}(v)$. 
\end{proof}
From \cite[Thm. 1.14]{Gar05}, we have that $\GD_{K}(U)$ is metrizable
with metric 
\begin{equation}
d_{2}(u,v)=\sum_{n=1}^{+\infty}2^{-n}\cdot\min\left\{ \mathcal{P}_{U_{\text{d}\eps^{n}}^{n}(0)}(u-v),1\right\} .\label{eq:d2Metric}
\end{equation}
Concerning \eqref{eq:d2Metric} we recall (see \cite{Gar05}) that
if $A\subseteq\GD_{K}(U)$ is $\Rtil$-absorbent, then, for all $u\in\GD_{K}(U)$,
we define 
\[
V_{A}(u):=\sup\left\{ b\in\R\mid u\in\diff{\eps}^{b}\cdot A\right\} 
\]
\begin{equation}
\mathcal{P}_{A}(u):=e^{-V_{A}(u)}.\label{eq:defGaugeP_A}
\end{equation}
The following result gives a metric which is equivalent to \eqref{eq:d2Metric}
but is defined by a simpler formula. 
\begin{prop}
\label{thm:GD_K-metrizable} Set $A_{n}:=U_{\text{\emph{d}}\eps^{n}}^{n}(0)$
for $n\in\N_{>0}$, and let $u\in\GD_{K}(U)$. Then 
\begin{enumerate}
\item \label{enu:V_Amn}$V_{A_{n}}(u)=v_{n}(u)-n$ 
\item \label{enu:d_e-metric} The map 
\[
d_{e}(u,v)=\sum_{n=1}^{+\infty}e^{\min\left[n-v_{n}(u-v),0\right]-n}.
\]
is a metric on $\GD_{K}(U)$ that is equivalent to $d_{2}$. 
\end{enumerate}
\end{prop}
\begin{proof}
\noindent Concerning \ref{enu:V_Amn}, we note that 
\[
\begin{split}u\in\diff{\eps}^{b}\cdot A_{n} & \Rightarrow\frac{u}{\text{d}\eps^{b}}\in U_{\text{d}\eps{}^{n}}^{n}(0)\Rightarrow\frac{\Vert u\Vert_{n}}{\text{d}\eps{}^{b+n}}\approx0\Rightarrow v\left(\diff{\eps}^{-b-n}\cdot\Vert u\Vert_{n}\right)\ge0\\
 & \Rightarrow b\le v_{n}(u)-n,
\end{split}
\]
so $V_{A_{n}}(u)\le v_{n}(u)-n$. Conversely, if we had $V_{A_{n}}(u)<v_{n}(u)-n$,
we could pick $V_{A_{n}}(u)<b<v_{n}(u)-n$. Then as above it would
follow that $\frac{\Vert u\Vert_{n}}{\text{d}\eps{}^{b+n}}\approx0$,
contradicting the definition of $V_{A_{n}}(u)$. This proves \ref{enu:V_Amn}.

In order to prove \ref{enu:d_e-metric}, we use \ref{enu:V_Amn} in
\eqref{eq:d2Metric}: $\mathcal{P}_{A_{n}}(u-v)=e^{-v_{n}(u-v)+n}$
and 
\begin{align*}
d_{2}(u,v) & =\sum_{n=1}^{+\infty}2^{-n}\cdot\min\left\{ e^{-v_{n}(u-v)+n},1\right\} =\\
 & =\sum_{n=1}^{+\infty}2^{-n}\cdot e^{\min\left[n-v_{n}(u-v),0\right]}\ge\\
 & \ge\sum_{n=1}^{+\infty}e^{\min\left[n-v_{n}(u-v),0\right]-n}=d_{e}(u,v).
\end{align*}
But we also have $e^{-n}\ge2^{-n-1}$ for all $n$, so that $d_{e}(u,v)\ge\frac{1}{2}d_{2}(u,v)$.
Using Prop.~\ref{prop:propValuationsUPnorm}, it is easily checked
that $d_{e}$ is a metric, which we have just proved to be equivalent
to $d_{2}$. 
\end{proof}

\section{\label{sec7}Completeness of \texorpdfstring{$\GD_{K}$}{GDK}}

In order to prove the completeness of $\GD_{K}(U)$, we generalize
the proof of \cite[Prop. 3.4]{Gar05} (based in turn on \cite{S})
to the present context. 
\begin{thm}
\label{thm:GD_Kcomplete}The space $\GD_{K}(U)$ with the sharp topology
is complete.\end{thm}
\begin{proof}
By Prop.~\ref{thm:GD_K-metrizable} $\GD_{K}(U)$ with the sharp
topology is metrizable. Hence, it suffices to consider a Cauchy sequence
$(u_{n})_{n\in\N}$ in this topology, i.e., 
\[
\forall q\in\R_{>0}\,\forall i\in\N\,\exists N\in\N\,\forall m,n\ge N:\ \Vert u_{n}-u_{m}\Vert_{i}<\diff{\eps}^{q}.
\]
Setting $i=q=k\in\N_{>0}$, this implies the existence of a strictly
increasing sequence $(n_{k})_{k\in\N}$ in $\N$ such that $\Vert u_{n_{k+1}}-u_{n_{k}}\Vert_{k}<\diff{\eps}^{k}$.
Hence picking any representatives $(u_{n\eps})$ of $u_{n}$ as in
Def.~\ref{def:cmptlySuppGSF} we have 
\[
\left[\max_{|\alpha|\le k}\sup_{x\in\R^{n}}\left|\partial^{\alpha}u_{n_{k+1},\eps}(x)-\partial^{\alpha}u_{n_{k},\eps}(x)\right|\right]<\left[\eps^{k}\right]\quad\forall k\in\N_{>0}.
\]
By Lemma \ref{lem:mayer} this yields that for each $k\in\N_{>0}$
there exists an $\eps_{k}$ such that $\eps_{k}\searrow0$ and 
\begin{equation}
\forall\eps\in(0,\eps_{k}):\ \max_{|\alpha|\le k}\sup_{x\in\R^{n}}\left|\partial^{\alpha}u_{n_{k+1},\eps}(x)-\partial^{\alpha}u_{n_{k},\eps}(x)\right|<\eps^{k}.\label{eq:maxSup-z_keps}
\end{equation}
Now set 
\begin{equation}
h_{k\eps}:=\begin{cases}
u_{n_{k+1},\eps}-u_{n_{k},\eps}\in\Coo(\R^{n},\R) & \text{ if }\eps\in(0,\eps_{k})\\
0\in\Coo(\R^{n},\R) & \text{ if }\eps\in[\eps_{k},1)
\end{cases}\label{eq:def-h_keps}
\end{equation}
\[
u_{\eps}:=u_{n_{0}\eps}+\sum_{k=0}^{\infty}h_{k\eps}\quad\forall\eps\in I.
\]
Since $\eps_{k}\searrow0$, for all $\eps\in I$ we have $\eps\notin(0,\eps_{k})$
for all $k\ge\bar{k}$, with $\bar{k}$ sufficiently big. Therefore,
$u_{\eps}=u_{n_{\bar{k}+1},\eps}\in\Coo(\R^{n},\R)$. In order to
prove that $(u_{\eps})$ defines a GSF of the type $U\to\Rtil$, take
$[x_{\eps}]\in U$ and $\alpha\in\N$. We claim that $\left(\partial^{\alpha}u_{\eps}(x_{\eps})\right)\in\R_{M}$.
Now if $p\in\N$ satisfies $|\alpha|\le p$, then for any $x\in\R^{n}$
we have 
\[
\left|\partial^{\alpha}u_{\eps}(x)\right|\le\left|\partial^{\alpha}u_{n_{p+1},\eps}(x)\right|+\sum_{k=p+1}^{\infty}\left|\partial^{\alpha}h_{k\eps}(x)\right|.
\]
From \eqref{eq:maxSup-z_keps} and \eqref{eq:def-h_keps} we get that
$\left|\partial^{\alpha}h_{k\eps}(x)\right|\le\eps^{k}$ for all $k\ge p+1$,
$x\in\R^{n}$ and all $\eps\in(0,1]$. Hence for $\eps\in(0,1]$,
$|\alpha|\le p$ and all $x\in\R^{n}$ we obtain 
\begin{equation}
\left|\partial^{\alpha}u_{\eps}(x)\right|\le\left|\partial^{\alpha}u_{n_{p+1},\eps}(x)\right|+\frac{\eps^{p+1}}{1-\eps}.\label{thisest}
\end{equation}
Inserting $x=x_{\eps}$ and noting that $(\partial^{\alpha}u_{n_{p+1},\eps}(x_{\eps}))\in\R_{M}$
proves our claim. Moreover, since all $(u_{n\eps})$ satisfy Def.~\ref{def:cmptlySuppGSF},
we also conclude from \eqref{thisest} that for any $\alpha$ and
any $[x_{\eps}]\in\exterior{K}$ we have $[\partial^{\alpha}u_{\eps}(x_{\eps})]=0$,
and hence the
GSF $\left[u_{\eps}(-)\right]|_{U}\in\GD_{K}(U)$. 

Finally, $\Vert u-u_{n_{p}}\Vert_{i}<\diff{\eps}^{p-1}$ for all $p\in\N_{>1}$
and all $i\le p$. This yields that $(u_{n_{k}})_{k}$ tends to $u$
in the sharp topology, and hence so does $(u_{n})$. 
\end{proof}

\section{\label{sec8}The space \texorpdfstring{$\GD$}{GD} as inductive
limit of \texorpdfstring{$\GD_{K}$}{GDK}}

In this section, we always assume that $U\subseteq\Rtil^{n}$ is a
non-empty strongly internal set. By Prop.~\ref{prop:GD-Rtil-Module}
and Prop.~\ref{prop:suffCondsInterlUnionIn-U}, this entails that
$\GD(U)$ is an $\Rtil$-module.

In order to define a natural topology on $\GD(U)$ we will employ
\cite[Thm. 1.18]{Gar05}, which we restate here for the reader's convenience: 
\begin{thm}
\label{thm:1.18Gar}Let $\mathcal{G}$ be an $\Rtil$-module, $(\mathcal{G}_{\gamma})_{\gamma\in\Gamma}$
be a family of locally convex topological $\Rtil$-modules and, for
each $\gamma\in\Gamma$, let $i_{\gamma}:\mathcal{G}_{\gamma}\ra\mathcal{G}$
be an $\Rtil$-linear map. Assume that 
\[
\mathcal{G}=\text{\emph{span}}\left(\bigcup_{\gamma\in\Gamma}i_{\gamma}\left(\mathcal{G}_{\gamma}\right)\right)
\]
and let $V\in\mathcal{V}$ if and only if $V\subseteq\mathcal{G}$,
$V$ is absolutely $\Rtil$-convex and $i_{\gamma}^{-1}(V)$ is a
neighborhood of $0$ in $\mathcal{G}_{\gamma}$ for all $\gamma\in\Gamma$.
Then each $V\in\mathcal{V}$ is $\Rtil$-absorbent and the topology
$\tau$ induced by the gauges $\left\{ \mathcal{P}_{V}\right\} _{V\in\mathcal{V}}$
(see \cite{Gar05} and Def.~\eqref{eq:defGaugeP_A}) is the finest
$\Rtil$-locally convex topology on $\mathcal{G}$ such that $i_{\gamma}$
is continuous for all $\gamma\in\Gamma$. Endowed with this topology,
$\mathcal{G}$ is called the inductive limit (colimit) of the spaces
$(\mathcal{G}_{\gamma})_{\gamma\in\Gamma}$ and we write $\mathcal{G}=\indlim\mathcal{G}_{\gamma}$. 
\end{thm}
Since $\GD(U)=\bigcup_{\emptyset\ne K\fcmp U}\GD_{K}(U)$, we may
therefore equip it with the inductive limit topology with respect
to the inclusions $\iota_{K}:\GD_{K}(U)\hookrightarrow\GD(U)$. We
call the resulting $\Rtil$-locally convex topology the \emph{sharp
topology} on $\GD(U)$. Hence 
\[
\GD(U)=\indlim\GD_{K}(U)\qquad(\emptyset\not=K\fcmp U).
\]
Henceforth we will denote the sharp topology on $\GD_{K}(U)$ by $\sigma_{K}(U)$
(or, for short, by $\sigma_{K}$). Also, the inductive limit topology
on $\GD(U)$ will be denoted by $\sigma(U)$ (or by $\sigma$). Setting
\[
U_{\text{d}\eps^{n}}^{n}(0):=\left\{ f\in\GD(U)\mid\frac{\Vert f\Vert_{n}}{\diff{\eps}^{n}}\approx0\right\} ,
\]
where $n\in\N_{>0}$, we obtain an $\Rtil$-absorbent and absolutely
$\Rtil$-convex subset of $\GD(U)$ such that $i_{K}^{-1}\left(U_{\text{d}\eps^{n}}^{n}(0)\right)=U_{\text{d}\eps^{n}}^{n}(0)\cap\GD_{K}(U)=U_{\text{d}\eps^{n}}^{n}(0,K)\in\sigma_{K}$.
Therefore, these sets generate a coarser topology than the sharp topology
$\sigma$. The proof is identical to that of Prop.~\ref{prop:U-rho-m-topology}
\ref{enu:U-rho-m-abs}.

From the (co-)universal property of inductive limits (\cite[Prop. 1.19]{Gar05})
we immediately conclude: 
\begin{prop}
\label{prop:GDasColimitGD_K}Let $\mathcal{H}$ be a locally convex
topological $\Rtil$-module. For each non-empty $K\fcmp U$, let $T_{K}:\GD_{K}(U)\ra\mathcal{H}$
be an $\Rtil$-linear and continuous map. Assume that $T_{K}(f)=T_{H}(f)$
if $f\in\GD_{K}(U)\cap\GD_{H}(U)$. Then there exists one and only
one map $T:\GD(U)\ra\mathcal{H}$ which is $\Rtil$-linear and continuous
and such that $T\circ\iota_{K}=T_{K}$ for all non-empty $K\fcmp U$. 
\end{prop}
The work \cite{Gar05} includes a detailed analysis of countable inductive
limits $\mathcal{G}=\indlim\mathcal{G}_{n}\ (n\in\N)$, where $(\mathcal{G}_{n})_{n\in\N}$
is increasing, $\mathcal{G}=\bigcup_{n\in\N}\mathcal{G}_{n}$, and
where the topology on $\mathcal{G}_{n}$ is that induced by $\mathcal{G}_{n+1}$.
Such inductive limits are called \emph{strict}. As in the case of
classical function spaces like $\mathcal{D}(\Omega)$ for $\Omega$
open in $\R^{n}$, the importance of strict inductive limits in the
theory of $\Rtil$-locally convex models ultimately stems from the
possibility of covering every open set $\Omega\subseteq\R^{n}$ by
a countable increasing family of compact sets. The key point in the
structure theory of strict inductive limits as above is that a countable
family $\left(\mathcal{G}_{n}\right)_{n\in\N}$ permits to define
recursively a family of neighborhoods of $0$. Using the latter, one
can prove that the topology on $\mathcal{G}$ induces on each $\mathcal{G}_{n}$
its given topology. We will show below that similar properties hold
for $\GD(U)$. We shall see that the assumption of $U$ being strongly
internal and sharply open are essential for this task. To begin with,
we prove a strengthening of Thm.~\ref{thm:GD_K-Rtil-Module}: 
\begin{prop}
\label{thm:GD_K-inducedTop}Let $H$, $K\fcmp U$ be non-empty sets,
with $H\subseteq K$. Then $\GD_{H}(U)$ is a topological subspace
of $\GD_{K}(U)$, i.e., 
\[
\sigma_{K}(U)|_{\GD_{H}(U)}=\sigma_{H}(U).
\]
\end{prop}
\begin{proof}
In this proof we will use the more precise notation $B_{\rho}^{m}(u,K)$
for balls (see Def.~\ref{def:ballsTopology}).

Let $V\in\sigma_{K}$. We claim that $V\cap\GD_{H}(U)\in\sigma_{H}$.
For each $u\in V\cap\GD_{H}(U)$ there exist $m\in\N$ and $\rho\in\Rtil_{>0}$
such that $B_{\rho}^{m}(u,K)\subseteq V$. But $B_{\rho}^{m}(u,H)\subseteq B_{\rho}^{m}(u,K)$
because $\GD_{H}(U)\subseteq\GD_{K}(U)$ and because the norm $\Vert-\Vert_{m}$
doesn't depend on $H,$ $K$. Therefore, $B_{\rho}^{m}(u,H)\subseteq V\cap\GD_{H}(U)$,
which proves our claim.

Conversely, if $W\in\sigma_{H}$, then we set 
\begin{equation}
V:=\bigcup\left\{ B_{\rho}^{m}(u,K)\mid u\in W\ ,\ m\in\N\ ,\ \rho\in\Rtil_{>0}\ ,\ B_{\rho}^{m}(u,H)\subseteq W\right\} \in\sigma_{K},\label{eq:defV}
\end{equation}
and we claim that $W=V\cap\GD_{H}(U)$. In fact, since $W\in\sigma_{H}$,
for all $u\in W$ we have $B_{\rho}^{m}(u,H)\subseteq W$ for some
$\rho$ and $m$. By \eqref{eq:defV} $B_{\rho}^{m}(u,K)\subseteq V$,
and so $u\in V\cap\GD_{H}(U)$. Vice versa, if $u\in V\cap\GD_{H}(U)$,
then $u\in B_{\rho}^{m}(v,K)$ for some $v\in W$, $m$, $\rho$,
such that $B_{\rho}^{m}(v,H)\subseteq W$. So $\Vert u-v\Vert_{m}<\rho$
and hence $u\in B_{\rho}^{m}(v,H)\subseteq W$. 
\end{proof}
We now show that the space $\GD(U)$ can be seen as a strict inductive
limit of a countable increasing family of subspaces $\GD_{K}(U)$.
Indeed, since $U$ is strongly internal, we can write $U=\sint{U_{\eps}}$
for some net $(U_{\eps})$ of subsets of $\R^{n}$. Since $U$ is
non-empty, by \cite[Thm. 8]{GKV}, fixing any $x=[x_{\eps}]\in U$
we obtain: 
\begin{equation}
\exists N\in\N\,\forall^{0}\eps:\ d(x_{\eps},U_{\eps}^{c})>\eps^{N}\ ,\ |x_{\eps}|\le\eps^{-N}.\label{eq:stronglyInternalModer}
\end{equation}
With $N$ as in \eqref{eq:stronglyInternalModer}, we define 
\begin{align*}
K_{j\eps}: & =\left\{ x\in\R^{n}\mid d(x,U_{\eps}^{c})\ge\eps^{j}\ ,\ |x|\le\eps^{-j}\right\} \\
K_{j}: & =[K_{j\eps}]\quad\forall j\in\N_{\ge N}.
\end{align*}
Using this notation, we have: 
\begin{thm}
\label{thm:GDasCountableColimit}If $N\in\N$ satisfies \eqref{eq:stronglyInternalModer}
for some $x\in U$, then 
\begin{enumerate}
\item \label{enu:GDasCountColimit-K_j-fcmp}$\emptyset\ne K_{j}\fcmp U$
for all $j\in\N_{\ge N}$ 
\item \label{enu:GDasCountColimit-K_j-increasing}$K_{j}\subseteq K_{j+1}$
for all $j\in\N_{\ge N}$ 
\item \label{enu:GDasCountColimit-unionK_j}$U=\bigcup_{j\ge N}K_{j}$ 
\item \label{enu:four} For every $\emptyset\ne K\fcmp U=\sint{U_{\eps}}$
there exists some $j\ge N$ such that $K\subseteq K_{j}$. 
\item \label{enu:GDasCountColimit-GDasCountColimit} $\GD(U)$ is the strict
inductive limit of the family $\GD_{K_{j}}(U)$, $j\ge N$: 
\[
\GD(U)=\indlim\GD_{K_{j}}(U)\quad(j\ge N).
\]

\end{enumerate}
\end{thm}
\begin{proof}
\ref{enu:GDasCountColimit-K_j-fcmp}, \ref{enu:GDasCountColimit-K_j-increasing}:
It follows immediately from the definition that each $(K_{j\eps})$
is compact and that $(K_{j\eps})_{\eps}$ is sharply bounded, so $K_{j}\fcmp U$.
Moreover, $K_{N}\ne\emptyset$ by \eqref{eq:stronglyInternalModer},
hence $K_{j}\ne\emptyset$ follows by \ref{enu:GDasCountColimit-K_j-increasing},
which again is immediate from the definition.

\ref{enu:GDasCountColimit-unionK_j}: If $x=[x_{\eps}]\in U=\sint{U_{\eps}}$,
then $d(x_{\eps},U_{\eps}^{c})>\eps^{j_{1}}$ for $\eps$ small and
for some $j_{1}\in\N_{>0}$. Since $(x_{\eps})$ is moderate, $|x_{\eps}|\le\eps^{-j_{2}}$
for $\eps$ small and some $j_{2}\in\N$. Setting $j:=\max(j_{1},j_{2},N)$
we hence have that $x\in K_{j}$.

In order to prove \ref{enu:four}, we need the following strengthening
of \cite[Thm. 11]{GKV}: 
\begin{lem}
\label{lem:funCmptAndStronglyInternal}Let $H=[H_{\eps}]\fcmp V=\sint{V_{\eps}}$,
then 
\begin{equation}
\exists j\in\N\,\forall[x_{\eps}]\in[H_{\eps}]\,\forall^{0}\eps:\ d(x_{\eps},V_{\eps}^{c})\ge\eps^{j}.\label{eq:conclLemfunCmptSint}
\end{equation}
\end{lem}
\begin{proof}[Proof of Lemma \ref{lem:funCmptAndStronglyInternal}]
Equation \eqref{eq:conclLemfunCmptSint} expresses that for all representatives
$(x_{\eps})\in\R_{M}^{n}$, if $x_{\eps}\in H_{\eps}$ for $\eps$
small, then $\forall^{0}\eps:\ d(x_{\eps},V_{\eps}^{c})\ge\eps^{j}$.

By contradiction, assume that 
\[
\forall j\in\N\,\exists(x_{j\eps})\in\R_{M}^{n}:\ \left(\forall^{0}\eps:\ x_{j\eps}\in H_{\eps}\right)\ ,\ \exists(\eps_{jk})_{k}\downarrow0\,\forall k:\ d(x_{j\eps_{jk}},V_{\eps_{jk}}^{c})<\eps_{jk}^{j}.
\]
By recursively applying this condition, we get that for all $j\in\N$
there exists a moderate $(x_{j\eps})$ and some $\eps_{j}\in(0,1]$
such that $x_{j\eps}\in H_{\eps}$ for $\eps\le\eps_{j}$ and 
\begin{equation}
d(x_{j\eps_{jk}},V_{\eps_{jk}}^{c})<\eps_{jk}^{j},\label{eq:dist-ejk}
\end{equation}
where $(\eps_{jk})_{k}\downarrow0$. Since $(\eps_{jk})_{k}\downarrow0$,
without loss of generality we can assume to have defined recursively
$(\eps_{j})_{j}$ so that $(\eps_{j})_{j}\downarrow0$ and $\eps_{j}>\eps_{jk_{j}}>\eps_{j+1}$
for some subsequence $(k_{j})_{j}\uparrow+\infty$. Set $x_{\eps}:=x_{j\eps}\in H_{\eps}$
for $\eps\in(\eps_{j+1},\eps_{j}]$, so that $x_{\eps_{jk_{j}}}=x_{j\eps_{jk_{j}}}$
for all $j$. Then $(x_{\eps})\in\R_{M}^{n}$ since $H$ is sharply
bounded and $x:=[x_{\eps}]\in H\subseteq\sint{V_{\eps}}$, which entails
\begin{equation}
\exists q\in\R_{>0}\,\forall^{0}\eps:\ d(x_{\eps},V_{\eps}^{c})>\eps^{q}.\label{eq:xInSint-V_eps}
\end{equation}
Therefore, for $j\in\N$ sufficiently big, \eqref{eq:xInSint-V_eps}
holds at $\eps=\eps_{jk_{j}}<1$ and $j>q$. Thus $d(x_{j\eps_{jk_{j}}},V_{\eps_{jk_{j}}}^{c})=d(x_{\eps_{jk_{j}}},V_{\eps_{jk_{j}}}^{c})>\eps_{jk_{j}}^{q}>\eps_{jk_{j}}^{j}$,
which contradicts \eqref{eq:dist-ejk}. 
\end{proof}
Continuing the proof of \ref{enu:four}, if $K\fcmp U$ is non-empty,
by applying Lemma \ref{lem:funCmptAndStronglyInternal} we obtain
\[
\exists j_{1}\in\N\,\forall[x_{\eps}]\in K\,\forall^{0}\eps:\ d(x_{\eps},U_{\eps}^{c})\ge\eps^{j_{1}}.
\]
On the other hand, sharp boundedness of $(K_{\eps})$, where $K=[K_{\eps}]$,
implies 
\[
\exists j_{2}\in\N\,\forall[x_{\eps}]\in K\,\forall^{0}\eps:\ |x_{\eps}|\le\eps^{-j_{2}}.
\]
Therefore, for $j:=\max(j_{1},j_{2},N)$, we get $K\subseteq K_{j}$,
hence \ref{enu:four}, and thereby also $\GD_{K}(U)\subseteq\GD_{K_{j}}(U)$
(using Thm.\ \ref{thm:GD_K-Rtil-Module}). It follows that $\GD(U)\subseteq\bigcup_{j\ge N}\GD_{K_{j}}(U)$.
The converse inclusion follows directly from \ref{enu:GDasCountColimit-K_j-fcmp}.

It remains to prove that the topology $\sigma$ on $\GD(U)$ coincides
with the inductive $\Rtil$-locally convex topology generated by $\left(\GD_{K_{j}}(U)\right)_{j\ge N}$.
Let us denote the latter topology by $\sigma'$. We have $\sigma\subseteq\sigma'$
by definition of the inductive topology and by Thm.~\ref{thm:GDasCountableColimit}.\ref{enu:GDasCountColimit-K_j-fcmp}.
To see that, conversely, $\sigma'\subseteq\sigma$ we show that for
every $K\fcmp U$ the inclusion $(\GD_{K}(U),\sigma_{K})\hookrightarrow(\GD(U),\sigma')$
is continuous (see Thm.~\ref{thm:1.18Gar}). Now given any $K\fcmp U$,
by what we have proved above there exists some $j\ge N$ with $K\subseteq K_{j}$.
But then Prop.~\ref{thm:GD_K-inducedTop} implies the continuity
of 
\[
(\GD_{K}(U),\sigma_{K})\hookrightarrow(\GD_{K_{j}}(U),\sigma_{K_{j}})\hookrightarrow(\GD(U),\sigma')
\]
and thereby our claim. 
\end{proof}
As a consequence of this result, we can now prove a series of corollaries
by applying the general theorems of \cite{Gar05} concerning countable
strict inductive limits. 
\begin{cor}
\label{cor:GD-inducesSigma_K}If $\emptyset\ne K\fcmp U=\sint{U_{\eps}}$,
then $\GD_{K}(U)$ is a topological subspace of $\GD(U)$, i.e., $\sigma(U)|_{\GD_{K}(U)}=\sigma_{K}(U)$.\end{cor}
\begin{proof}
According to Thm.~\ref{thm:GDasCountableColimit} \ref{enu:four}
we may pick $j\ge N$ such that $K\subseteq K_{j}$. By \cite[Prop. 1.21]{Gar05},
$\GD_{K_{j}}(U)$ carries the trace topology of $\GD(U)$. Since,
in turn, $\GD_{K}(U)$ is a topological subspace of $\GD_{K_{j}}(U)$
by Prop.~\ref{thm:GD_K-inducedTop}, the claim follows.\end{proof}
\begin{cor}
\label{cor:GD-separated} $\GD(U)$ is separated.\end{cor}
\begin{proof}
This follows from Thm.~\ref{thm:sharpTop} \ref{enu:LCTSep} and
\cite[Cor. 1.24]{Gar05}.\end{proof}
\begin{lem}
\label{lem:GD_H-closed-GD_K}If $\emptyset\ne H$, $K\fcmp U$ and
$H\subseteq K$, then $\GD_{H}(U)$ is closed in $\GD_{K}(U)$.\end{lem}
\begin{proof}
This is immediate from Prop.~\ref{thm:GD_K-inducedTop} and Thm.~\ref{thm:GD_Kcomplete}.\end{proof}
\begin{cor}
\label{cor:GD_K-closed-GD}If $\emptyset\ne K\fcmp U$, then $\GD_{K}(U)$
is closed in $\GD(U)$.\end{cor}
\begin{proof}
This follows from Thm.~\ref{thm:GD_Kcomplete}, Cor.~\ref{cor:GD-separated},
and Cor.~\ref{cor:GD-inducesSigma_K}. \end{proof}
\begin{cor}
\label{cor:boundedInGD}Let $B\subseteq\GD(U)$, then $B$ is bounded
in $\GD(U)$ if and only if there exists a non-empty $K\fcmp U$ such
that $B$ is bounded in $\GD_{K}(U)$.\end{cor}
\begin{proof}
$\Rightarrow$: This follows from \cite[Thm. 1.26]{Gar05} and Lemma
\ref{lem:GD_H-closed-GD_K}.

$\Leftarrow$: If $B$ is bounded in $\GD{}_{K}(U)$, then \cite[Lem. 1.27]{Gar05}
yields 
\begin{equation}
\forall(u_{n})_{n}\in B^{\N}\,\forall(\lambda_{n})_{n}\in\Rtil^{\N}:\ \lambda_{n}\to0\text{ in }\Rtil\then\lambda_{n}u_{n}\to0\text{ in }\GD_{K}(U).\label{eq:boundedBySequences}
\end{equation}
Pick $j\ge N$ such that $K\subseteq K_{j}$. Since generalized norms
$\Vert-\Vert_{m}$ do not depend on $K$, $K_{j}$, condition \eqref{eq:boundedBySequences}
holds also in $\GD_{K_{j}}(U)\supseteq\GD_{K}(U)$. Therefore, from
\cite[Lem. 1.27]{Gar05} we get that $B$ is bounded in $\GD_{K_{j}}(U)$
and \cite[Thm. 1.26]{Gar05} yields that $B$ is bounded in $\GD(U)$. 
\end{proof}
A similar proof applies to this corollary, which is a consequence
of \cite[Cor. 1.29]{Gar05}: 
\begin{cor}
\label{cor:sequenceConvergenceInGD}Let $(u_{n})_{n}\in\GD(U)^{\N}$,
then $u_{n}\to0$ in $\GD(U)$ if and only if there exists a non-empty
$K\fcmp U$ such that $u_{n}\in\GD_{K}(U)$ and $u_{n}\to0$ in $\GD_{K}(U)$. 
\end{cor}
Finally, from \cite[Thm. 1.32]{Gar05}, Lemma \ref{lem:GD_H-closed-GD_K}
and Thm.~\ref{thm:GD_Kcomplete} we obtain: 
\begin{cor}
\label{cor:GDcomplete}$\GD(U)$ with the sharp topology is complete. 
\end{cor}
Using Lemma \ref{lem:funCmptAndStronglyInternal}, we can also show
that any compactly supported generalized smooth function $f\in\GD_{K}(U,Y)$
on a sharply open set $U\subseteq\Rtil^{n}$ is defined by a net $(u_{\eps})$
of smooth functions which are compactly supported in an arbitrarily
small extension $\left[K_{\eps}+\overline{\Eball_{\eps^{a}}(0)}\right]$
of $K=[K_{\eps}]$. We recall that in this section we are assuming
that $U$ is a strongly internal set.
\begin{thm}
\label{thm:cmptlySuppDefNet}Let $\emptyset\ne K\fcmp U$. Let $Y\subseteq\Rtil^{d}$,
$f\in\GD_{K}(U,Y)$ and $K=[K_{\eps}]\fcmp\Rtil^{n}$. Let $j\in\N$
be as in \eqref{eq:conclLemfunCmptSint} and $a\in\R$ such that $a\ge j$.
Then there exist nets $(u_{\eps})$, $(H_{\eps})$ such that: 
\begin{enumerate}
\item $[H_{\eps}]\fcmp U$ 
\item $(u_{\eps})$ defines $f$ and $u_{\eps}\in\D_{H_{\eps}}(\R^{n},\R^{d})$
for all $\eps$ 
\item $H_{\eps}\subseteq K_{\eps}+\overline{\Eball_{\eps^{a}}(0)}$ for
all $\eps$. 
\end{enumerate}
\end{thm}
\begin{proof}
Let $U=\sint{U_{\eps}}$, where each $U_{\eps}\subseteq\R^{n}$ is
an open set (cf.~\cite[Cor. 9]{GKV}), and let $(v_{\eps})$ satisfy
Def.~\ref{def:cmptlySuppGSF} for $f$ and $K=[K_{\eps}]$. By \cite[Thm. 11]{GKV},
we can assume $K_{\eps}\subseteq U_{\eps}$ for all $\eps$. Let $L_{\eps}:=K_{\eps}+\overline{\Eball_{\eps^{a}/2}(0)}$,
and denote by $\chi_{L_{\eps}}$ the characteristic function of $L_{\eps}$.
Let $\psi\in\mathcal{D}(\Eball_{1}(0))$ have unit integral and set
$\psi_{\eps}:=(\eps^{a}/2)^{-n}\psi(2x/\eps^{a})$. Then $(\psi_{\eps})\in\mathcal{E}_{M}^{s}(\R^{n})$,
and $(\chi_{L_{\eps}}*\psi_{\eps^{a}})|_{K_{\eps}}=1$. Set $u_{\eps}:=(\chi_{L_{\eps}}*\psi_{\eps})\cdot v_{\eps}$.
Then $(\chi_{L_{\eps}}*\psi_{\eps})$ defines a GSF of the type $\Rtil^{n}\ra\Rtil$
and hence $(u_{\eps})$ defines a GSF of the type $U\ra\Rtil^{d}$.
Moreover, $H_{\eps}:=\text{supp}(u_{\eps})\subseteq\text{supp}(\chi_{L_{\eps}}*\psi_{\eps})\subseteq K_{\eps}+\overline{\Eball_{\eps^{a}}(0)}$.
Since $a\ge j$, we have $[H_{\eps}]\subseteq\sint{U_{\eps}}=U$.

It remains to prove that $f(x)=[u_{\eps}(x_{\eps})]$ for all $x=[x_{\eps}]\in U$.
Suppose this was not the case. Then there would exist some $y=[y_{\eps}]\in U$,
some $b>0$ and a sequence $\eps_{k}\searrow0$ such that 
\begin{equation}
|u_{\eps_{k}}(y_{\eps_{k}})-v_{\eps_{k}}(y_{\eps_{k}})|\ge\eps_{k}^{b}\label{last}
\end{equation}
for all $k\in\N$. By Lemma \ref{lem:twoCases}, we may without loss
of generality assume that either $y_{\eps}\in K_{\eps}$ for all $\eps$
or that $y\in\exterior{K}$. In the first case, $v_{\eps}(y_{\eps})=u_{\eps}(y_{\eps})$
for all $\eps$, contradicting \eqref{last}. In the second case,
\[
\left|v_{\eps}(y_{\eps})-u_{\eps}(y_{\eps})\right|\le2|v_{\eps}(y_{\eps})|.
\]
Since $\left(v_{\eps}(y_{\eps}))\right)$ is negligible, we again
arrive at a contradiction to \eqref{last}. 
\end{proof}
Assume that we have an operator $I:\D(\R^{n})\ra\R$ with the property
that if $(u_{\eps})$ and $(v_{\eps})$ define $f\in\GD(U)$, where
$u_{\eps}$, $v_{\eps}\in\D(\R^{n})$, then $\left[I(u_{\eps})\right]=\left[I(v_{\eps})\right]\in\Rtil$.
Then Thm.~\ref{thm:cmptlySuppDefNet} permits to extend $I$ to the
whole of $\GD(U)$.

Using this result, we can now prove the extension of property \ref{enu:GDKSubGlobal}
of Thm.~\ref{thm:globallyDefGDK} to arbitrary codomains $Y\subseteq\Rtil^{d}$:
\begin{thm}
\label{thm:extensionAndCodomains}Let $\emptyset\ne K\fcmp U$, $Y\subseteq\Rtil^{d}$
and $f\in\GD_{K}(U,Y)$, then $\exists!\bar{f}\in\GD^{\text{\emph{g}}}(K,Y):\ \bar{f}|_{K}=f|_{K}$.\end{thm}
\begin{proof}
We only have to prove that $\bar{f}(x)\in Y$ for all $x\in\Rtil^{n}$.
Let $(u_{\eps})$ and $(H_{\eps})$ as in Thm.~\ref{thm:cmptlySuppDefNet}.
By Thm.~\ref{thm:extensionCmptlySupportedGSF}, we have that $(u_{\eps})$
also defines  $\bar{f}$. For each $\eps$ pick any point $h_{\eps}\in\partial H_{\eps}$
and set $\bar{x}_{\eps}:=x_{\eps}$ if $x_{\eps}\in H_{\eps}$, and
$\bar{x}_{\eps}:=h_{\eps}$ otherwise. Therefore $\bar{x}:=[\bar{x}_{\eps}]\in[H_{\eps}]\subseteq U$.
Then, if $x_{\eps}\notin H_{\eps}$, $|u_{\eps}(x_{\eps})-u_{\eps}(\bar{x}_{\eps})|=|u_{\eps}(x_{\eps})-u_{\eps}(h_{\eps})|=0$
because $u_{\eps}\in\D_{H_{\eps}}(\R^{n},\R^{d})$. Thus $\bar{f}(x)=f(\bar{x})\in Y$.
\end{proof}

\section{Conclusions and Further developments}

The notion of functionally compact set we introduced in the present
work permits to show that compactly supported GSF are close analogues
of classical compactly supported smooth functions. In particular,
their functional analytic properties parallel those of the test function
space of distribution theory. At the same time, for suitable $K$,
the space $\GD_{K}(\Rtil^{n})$ contains extensions to all CGF $\gs(\Omega)$
and hence also all Schwartz distributions.

The theory developed here opens the door to addressing several central
topics in the theory of nonlinear generalized functions from a new
angle. As indicated after Thm.~\ref{thm:cmptlySuppDefNet}, a direct
generalization of the integral of compactly supported functions to
compactly supported GSF is feasible. An immediate application of this
lies in a theory of integration for GSF that we hope will allow to
harmonize the Schwartz view of generalized functions as functionals
with that prevalent in Colombeau theory of considering generalized
functions as pointwise maps. Our approach will take inspiration from
Garetto's very fruitful duality theory of locally convex $\Ctil$-modules
\cite{Gar05b,GH,Garfund}.

A further natural development of the present article goes in the direction
of a generalization to suitable types of asymptotic gauges (see \cite{Gio-Nig14,Gio-Lup14})
and hence to the full and the diffeomorphism invariant Colombeau algebras.

Moreover, one can ask whether $\Rtil$-valued generalized norms in
$\GD_{K}(U)$ permit to generalize results from classical analysis,
like a Picard-Lindelöf theorem for ODE with a GSF right hand side,
or a Hahn-Banach theorem for functionals $I:\GD_{K}(U)\ra\Rtil$ defined
by diffeologically smooth functionals (see \cite{Gio-Wu14}) of the
type $I_{\eps}:\D_{K_{\eps}}(U_{\eps})\ra\R$, analogously to the
way a GSF is defined by a net of smooth functions. \medskip{}

\textbf{Acknowledgment:} We would like to thank H.~Vernaeve for helpful
discussions and the anonymous referee for several suggestions that
have led to considerable improvements in Sec.~\ref{sec3}.

\end{document}